\documentclass[12pt,reqno]{amsart}
\usepackage{a4}
\usepackage{amsmath}
\usepackage{amssymb}
\usepackage{amscd}                  
\usepackage{verbatim}
\usepackage[latin1]{inputenc}
\usepackage{graphicx}               

\newtheoremstyle{fancy}{}{}{\itshape}{}{\textsc\bgroup}{.\egroup}{ }{}
\newtheoremstyle{fancy2}{}{}{\rm}{}{\textsc\bgroup}{.\egroup}{ }{}

\newcounter{intro}

\newtheorem{theo}[intro]{Theorem}

\newcounter{theorem}
\theoremstyle{fancy}
\newtheorem{cor}[theorem]{Corollary}
\newtheorem{lem}[theorem]{Lemma}
\newtheorem{thm}[theorem]{Theorem}

\newtheorem{sublem}{Sublemma}

\theoremstyle{fancy2}           

\newtheorem{rem}[theorem]{Remark}

\newcommand{\cref}[1]{Corollary~\ref{#1}}

\newcommand{\lref}[1]{Lemma~\ref{#1}}
\newcommand{\lrefs}[2]{Lemmas~\ref{#1} and~\ref{#2}}

\newcommand{\tref}[1]{Theorem~\ref{#1}}
\newcommand{\sref}[1]{Section~\ref{#1}}
\newcommand\C{\mathbb C}
\newcommand\I{\mathbb{I}}
\newcommand\N{\mathbb{N}}
\newcommand\R{\mathbb R}
\newcommand\Z{\mathbb{Z}}
\newcommand{\Sphere}{\mathbb S} 

\newcommand\mcV{\mathcal V}  \newcommand\mcC{\mathcal C}
\newcommand\mcU{\mathcal U}   \newcommand\mcL{\mathcal L}
  \newcommand\mcH{\mathcal H}
\newcommand\mcA{\mathcal A}  \newcommand\mcM{\mathcal M}

\newcommand{\eps}{\varepsilon} 
\newcommand{\e}{\mathrm e} 
\renewcommand{\phi}{\varphi}   
\newcommand{\im}{\mathrm i} 

\newcommand{\wt}{\widetilde}           
\newcommand {\qf}[1]{\mathfrak{#1}}    

\newcommand{\id}{\operatorname{id}}

\newcommand{\Ker}{\operatorname{Ker}}

\newcommand{\spec}{\operatorname{spec}}
\newcommand{\dom}{\operatorname{dom}}

\newcommand{\vol}{\operatorname{vol}}

\newcommand{\Sobsymb}[1][{}]  {\mathsf H^{#1}} 
\newcommand{\Contsymb}[1][{}] {\mathsf C^{#1}} 
\newcommand{\Lsymb}[1][{}]    {\mathsf L^{#1}} 

\newcommand{\Cont} [2][{}]{\Contsymb[#1] ({#2})} 

\newcommand{\Contc}[2][{}]{\Contsymb[#1]_{\mathrm c}({#2})}
\newcommand{\Ci} [1]    {\Cont[\infty]{#1}}      
\newcommand{\Ccic}[1]   {\Contc[\infty]{#1}}

\newcommand{\Lsqr}[1]{\Lsymb[2]({#1})} 

\newcommand{\Sob}[2][1]{\Sobsymb[#1] ({#2})} 

\newcommand{\Sobloc}[2][1]{\Sobsymb[#1]_{\mathrm{loc}}({#2})}
\newcommand{\bigSob}[2][1]{\Sobsymb[#1] \bigl({#2}\bigr)} 

\newcommand{\norm}[2][{}]{\|{#2}\|_{{#1}}}    

\newcommand{\iprod}[3][{}]{\langle{#2},{#3}\rangle_{#1}}  
\newcommand{\bigiprod}[3][{}]{\bigl\langle{#2},{#3}\bigr\rangle_{#1}}
\newcommand{\Bigiprod}[3][{}]{\Bigl\langle{#2},{#3}\Bigr\rangle_{#1}}

\newcommand{\set}[2]{\{ \, #1  \, ; \, #2 \, \} } 
\newcommand{\bigset}[2]{\bigl\{ \, #1 \, ; \, #2 \, \bigr\} }

\newcommand{\map}[3]{ #1 \colon #2 \longrightarrow #3 } 

\newcommand{\clo}[1]{\overline{{#1}}} 

\newcommand{\Neu}{{\mathrm N}}              
\newcommand{\Dir}{{\mathrm D}}              

\begin{document}
\title[Gaps in the differential forms spectrum on cyclic coverings]
{Gaps in the differential forms spectrum on cyclic coverings}
      
\date{\today}

\author{Colette ANN\'E} 
\address{Laboratoire de Math\'ematiques Jean
  Leray, Universit\'e de Nantes, CNRS, Facult\'e des Sciences, BP 92208,
  44322 Nantes, France} 
\email{anne@math.univ-nantes.fr}

\author{Gilles CARRON}
\address{Laboratoire de Math\'ematiques Jean
  Leray, Universit\'e de Nantes, CNRS, Facult\'e des Sciences, BP
  92208, 44322 Nantes, France}
\email{carron@math.univ-nantes.fr}

\author{Olaf POST}
\address{Institut f\"ur Mathematik,
         Humboldt-Universit\"at zu Berlin,
         Rudower Chaus\-see~25,
         12489 Berlin,
         Germany}
\email{post@math.hu-berlin.de}
\date{\today}

\begin{abstract}
  We are interested in the spectrum of the Hodge-de Rham operator on a
  $\Z$-covering $X$ over a compact manifold $M$ of dimension $n+1$. Let
  $\Sigma$ be a hypersurface in $M$ which does not disconnect $M$ and such 
  that $M-\Sigma$ is a fundamental domain of the covering. If the cohomology 
  group $H^{n/2}(\Sigma)$ is trivial, we can construct for each $N \in \N$ a 
  metric $g=g_N$ on $M$, such that the Hodge-de Rham operator on the covering 
  $(X,g)$ has at least $N$ gaps in its (essential) spectrum. If
  $H^{n/2}(\Sigma) \ne 0$, the same statement holds true for the
  Hodge-de Rham operators on $p$-forms provided $p \notin \{n/2,n/2+1\}$. 
\end{abstract}

\maketitle

\section{Introduction}
\label{sec:intro}
A common feature of periodic operators is its \emph{band-gap} nature
of the spectrum. It is natural to ask how we can create gaps between
the bands of the spectrum. Here we will extend the analysis done by
the third named author in~\cite{P} to the Hodge-de Rham operator on
\emph{forms}. However, there are topological obstructions for the
existence of gaps in the spectrum of the Hodge-de Rham operator. The
following \tref{thmA} is a direct consequence of
\cite[Theorem~0.1]{car}:
\begin{theo}
  \label{thmA}
  Let $(M^{4k+1},g)$ be a compact oriented Riemannian manifold. Assume
  that $\Sigma\subset M$ is an oriented hypersurface, with
  \emph{non-zero signature} and not disconnecting $M$. Let
  $\Z\rightarrow\wt M \rightarrow M$ be the cyclic covering associated
  to $\Sigma$, then for any complete Riemannian metric on $\wt M$
  (periodic or not) the spectrum of the Hodge-de Rham Laplacian on
  $\wt M$ is ${[}0,\infty{[}$.
\end{theo}
The result we present here has also a topological restriction:
\begin{theo}
  Assume that $\Sigma^n\subset M^{n+1}$ is a hypersurface in a compact
  manifold $M$ and assume that $\Sigma$ does not disconnect $M$. Let
  $\Z\rightarrow\wt M\rightarrow M$ be the cyclic covering associated
  to $\Sigma$.

  If $p\ne n/2$ and $p \ne n/2+1$, then there is a family of periodic
  Riemannian metrics $g_{\eps}$ on $\wt M$ such that the spectrum of
  the Hodge-de Rham Laplacian acting on $p$-forms has $N_\eps$ gaps
  with $\lim_{\eps\to 0}N_\eps=+\infty$.

  If $p=n/2$ or $p=n/2+1$, the same conclusion holds provided that the
  $(n/2)$-Betti number of $\Sigma$ vanishes, i.e.,
  $b_{n/2}(\Sigma)=0$.
\end{theo}
Our result is obtained through a convergence result of the
differential form spectrum which generalises the study of the first
author and B.~Colbois \cite{AC2}. The family of metrics $g_\eps$ is
defined on $M$ as follows: outside a collar neighbourhood of $\Sigma,$
the metric is independent of $\eps$ and on this collar neighbourhood of
$\Sigma$ the Riemannian manifold $(M,g_\eps)$ is isometric to the
union of two copies of the truncated cone $([\eps,1]\times \Sigma,
dr^2+r^2h),$ where $h$ is a fixed Riemannian metric on $\Sigma,$ and of a
thin handle $[0,L]\times \Sigma$ endowed with the Riemannian metric
$dr^2+\eps^2 h$.

\begin{figure}[h]
  \centering
  \begin{picture}(0,0)
    \includegraphics{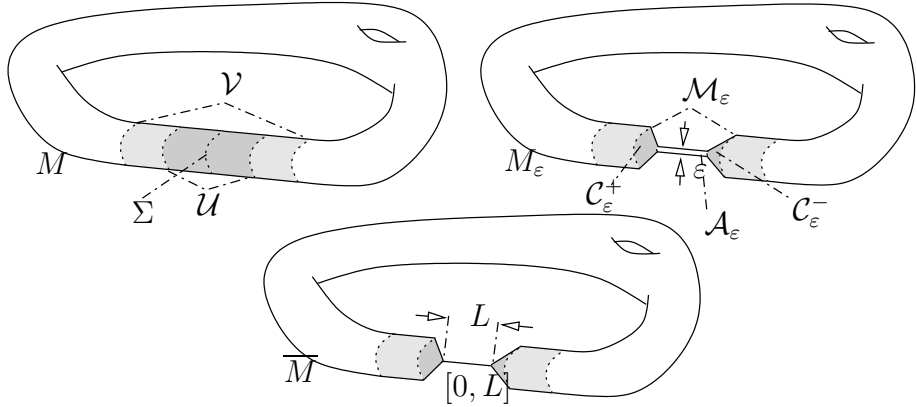}
  \end{picture}
  \setlength{\unitlength}{4144sp}
  \begin{picture}(5447,2407)(152,-1676) 
    \put(203,-310){$M$}
    \put(1310,174){$\mcV$}
    \put(1175,-553){$\mcU$}
    \put(781,-600){$\Sigma$}
    \put(3006,-289){$M_\eps$}
    \put(4071,138){$\mcM_\eps$} 
    \put(4140,-332){$\eps$}
    \put(4736,-581){$\mcC_\eps^-$}
    \put(4206,-669){$\mcA_\eps$}
    \put(3500,-487){$\mcC_\eps^+$} 
    \put(1689,-1552){$\overline M$}
    \put(2800,-1216){$L$}
    \put(2649,-1639){$[0,L]$}
  \end{picture}
  \caption{\small Construction of the manifold $M_\eps$ and the limit
    manifold $\overline M.$ We start with a manifold $M$ having
    product structure on $\mcU$. The cones on $M_\eps$ have length
    $1-\eps$, and the handle has length $L$ and radius $\eps$. The
    limit consists of the manifold $\overline M$ with two conical
    singularities, and the line segment $[0,L]$.}
  \label{fig:cone-mfd}
\end{figure}

Geometrically, the manifold $(M,g_\eps)$ is converging in the
Gromov-Hausdorff topology to the union of a manifold $(\overline M,
\overline g)$ with two conical singularities and of a segment of
length $L$ joining the two singularities.  On $(\overline M, \overline
g)$, the operator $D:=d+d^*$, a priori defined on the space of smooth
forms with support in the regular part of $\overline M$, is not
necessary essentially self adjoint. After the pioneering work~\cite{C}
of J.~Cheeger dealing with the Friedrichs extension $D_{\max}\circ
D_{\min}$, the closed extensions of $D$ have been studied carefully
(see for instance~\cite{BS}, \cite{L}, \cite{seeley:92}
and~\cite{HM}).

Denote by $\sigma^\Dir=\set{(\pi k/L)^2}{k = 1,2,\dots}$ the Dirichlet
spectrum of the Laplacian on functions on the interval $[0,L]$ and
similarly by $\sigma^\Neu := \sigma^\Dir \cup \{0\}$ the Neumann
spectrum.  Our main theorem is the following:
\begin{theo}
  \label{thmC} 
  Suppose, in the case when $n$ is even, that the cohomology group
  $H^{n/2}(\Sigma)=0$. The spectrum of the Hodge-de Rham operator
  acting on $p$-forms of the manifold $(M,g_\eps)$ converges to the
  spectrum $\sigma_p$ of the limit problem, where $\sigma_p$ is given
  as follows:
  \begin{description}
  \item[$\mathbf{p < (n+1)/2}$] The limit spectrum $\sigma_p$ is the
    union of the spectrum of the operator $D_{\max}\circ D_{\min}$ on
    $p$-forms on $\overline M$, the Neumann spectrum $\sigma^\Neu$
    with multiplicity $\dim H^{p-1}(\Sigma)$ and the Dirichlet
    spectrum $\sigma^\Dir$ with multiplicity $\dim H^{p}(\Sigma)$.
  \item[$\mathbf{p > (n+1)/2}$] The limit spectrum $\sigma_p$ is the
    union of the spectrum of the operator $D_{\max}\circ D_{\min}$ on
    $p$-forms on $\overline M$, the Neumann spectrum $\sigma^\Neu$
    with multiplicity $\dim H^{p}(\Sigma)$ and the Dirichlet spectrum
    $\sigma^\Dir$ with multiplicity $\dim H^{p-1}(\Sigma)$,
  \item[$\mathbf{p = (n+1)/2}$] The limit spectrum $\sigma_p$ is the
    union of the spectrum of the operator $D_{\min}\circ D_{\max}$
    on $p$-forms on $\overline M$, and the Neumann spectrum $\sigma^\Neu$
    with multiplicity $\dim H^{p}(\Sigma)\oplus\dim H^{p-1}(\Sigma).$ 
\end{description}
\end{theo}
\subsection*{Remarks.}  Our \tref{sansgap} gives also a convergence
result in the case when $n$ is even, the $(n/2)$-cohomology group of
$\Sigma$ is non-trivial and $p=n/2$ or $p=(n+1)/2$. In this case the
limit spectrum is obtained by a coupled problem between the manifold
$\overline M$ and the line segment. Consequently, the result of
\tref{sansgap} does not help for the determination of the spectrum on
the periodic manifold: The spectrum depends in fact on the spectral
flow (see~\cite[p.93]{APS3} for a definition) of the family of
operators defined by the Floquet parameter.

We remark also that the presence of the handle influences the
definition of the limit problem on the manifold $\overline M$, namely
in the case ${p = (n+1)/2}$ where in fact the operator $D_{\min}\circ
D_{\max}$ appears.  If the handle is not present (i.e.~$L=0$), the
index of the Gau\ss -Bonnet operator in this situation has been
studied by R.~Seeley in~\cite{seeley:92}, and the convergence of the
spectrum of the Hodge-de Rham operator acting on $p$-forms by
P.~Macdonald (\cite{Mac}), and next by R.~Mazzeo and J.~Rowlett
(\cite{Maz,Row}).  The result is that, {\em with the topological
  hypothesis $H^{n/2}(\Sigma)=0,$ this spectrum converges to the
  spectrum of the Friedrichs extension $D_{\max}\circ D_{\min}$ of the
  Hodge-de Rham operator on $\overline M$ for any degree $p$.}  This
fact can be recovered by our analysis.
  
Finally, our work has also an extension to the Dirac operator: there is an
analogue of \tref{thmA} due to J.~Roe for the Dirac operator
(\cite{Roe}). On the other  hand, if we consider a compact spin manifold 
$M^{n+1}$ and an oriented hypersurface $\Sigma$ with trivial 
$\widehat A$-invariant or trivial $\alpha$-invariant, then the recent work 
of B.~Ammann, M.~Dahl and E.~Humbert~\cite{ADH} provides a Riemannian metric 
$h$ on $\Sigma$ with no harmonic spinors. Then we can scale this metric so 
that its associated Dirac operator on $\Sigma$ has no eigenvalue in a large 
symmetric interval. Then  our construction also applies in this case, 
and gives, with J.~Roe's results, the following
\begin{theo}\label{thmD}
  Assume that $\Sigma^n\subset M^{n+1}$ is an oriented hypersurface in
  a compact spin manifold $M$, which does not disconnect $M$, and
  consider $\Z\rightarrow \wt M\rightarrow M,$ the associated cyclic
  covering.  Then there is a family $g_\eps$ of periodic Riemannian
  metrics on $\wt M$, whose Dirac operator has a large number of gaps
  in its spectrum {\em if and only if} $\widehat A (\Sigma)=0 ,$ in the case
  $n=4k,$ or $\alpha(\Sigma)=0,$ in the case $n=8k+1$ or $n=8k+2$.
\end{theo}

Recall that the spin cobordism $\alpha$-invariant satisfies
$\alpha(\Sigma)\in \Z /2\Z$. This last result can be compared with the
recent one of D.~Ruberman and N.~Saveliev. Indeed they prove in
\cite[Theorem 2]{RS} that, the Dirac operator on a cyclic covering
$\wt M \rightarrow M$ is invertible for a generic set of $\Z$-periodic
metric , if and only if $\alpha_{n+1}(M)=0$ and
$\alpha_{n}(\Sigma)=0.$ The topological invariant $\alpha_{n}(X)$ for
a closed manifold $X$ of dimension $n$ is defined as an elemant of
$KO_n$, and we have
  \begin{equation*}
    \alpha_n(X) =
    \begin{cases}
      \widehat A (X), & \text{if $n=8k$,}\\
      \widehat A (X)/2, & \text{if $n=8k+4$,}\\
      \alpha(X), & \text{if $n \in\{8k+1,8k+2\}$,}\\
      0, & \text{otherwise}.
    \end{cases}
  \end{equation*}
  They use also the construction of B.~Ammann, M.~Dahl and
  E.~Humbert~\cite{ADH}.  In particular, the results of Ruberman and
  Saveliev imply that generically, the first band of the spectrum of
  the Dirac operator does not touch $0$; it is not a result about the
  presence of many gaps in the spectrum.

  It is tempting to ask whether an equivalence as in \tref{thmD} also
holds for the Hodge-de Rham operator, but we have no guess about the
validity of such an extension. We think that it is an interesting
question and we intend to consider this question in a future work.

The paper is organised as follows: In the next section, we fix the
geometric setting for the quotient manifold $M$, namely the family of
metrics $g_\eps$. In \sref{sec:lap} we describe the Hodge-de Rham
operator in natural coordinates on the collar neighbourhood of
$\Sigma$. In \sref{sec:asymp} we provide basic estimates on a sequence
of eigenforms used in the main convergence result, which will be
presented in \sref{sec:limit}. In \sref{sec:covering} we deduce the
existence of spectral gaps and in \sref{sec:harm.forms} we discuss
the possible appearance of small eigenvalues in the setting of
\tref{thmC}.
\subsection*{Acknowledgements}
This work began with a one month visit of O.~Post at the University of
Nantes; O.~Post would like to thank for this invitation. G. Carron
thanks R. Mazzeo for useful discussions. We thank the referee for
drawing our attention to the work of Ruberman and Saveliev~\cite{RS}.
\section{The geometric set-up }
\label{sec:geo}

In this section, we explain the construction of the deforming family
of metrics~$g_\eps$. We assume that $M$ is a compact manifold of
dimension $n+1$ and that $\Sigma$ is a compact hypersurface in $M$
which does not disconnect the manifold (note that this hypothesis,
needed for the construction of a \emph{connected} periodic manifold,
does not play any role in the proof of the preliminary \tref{thmC}).
We choose a metric $g$ on $M$ such that there exists a collar
neighbourhood $\mcU={]}{-}2,2{[} \times \Sigma$ of $\Sigma$ where $g$
is of the form $dt^2+h$ for a (fixed) metric $h$ on $\Sigma$.

For $\eps \in {]}0,1]$, we construct a family of continuous, piecewise
smooth metrics $g_\eps$ on $M$ as follows:
\begin{itemize}
\item Outside the collar neighbourhood $\mcV:={]}-1,1{[} \times \Sigma
  \subset \mcU$, we do not change the metric, i.e. $g_\eps=g$ on $M\setminus \mcV$.
\item On the collar neighbourhood $\mcV$, the metric is chosen in such a way that
  $(\mcV,g_\eps)$ is isometric to the union $\mcM_\eps = \mcC_\eps^-
  \cup \mcA_\eps\cup \mcC_\eps^+$, where $\mcC_\eps^\pm$ are cones
  ${]}\eps,1{[} \times \Sigma$ endowed with the metric $dt^2+t^2 h$
  and with distinct orientation, and where $\mcA_\eps$ is the handle
  ${]}0,L{[} \times \Sigma$ endowed with the metric $dt^2+\eps^2 h$.
  Using a coordinate $\tau$ on all three parts, $(\mcM_\eps,g_\eps)$ is
  isometric to ${]}{-}(L/2+1-\eps), L/2+1-\eps{[} \times \Sigma$,
  endowed with the warped product metric $d\tau^2+\rho_\eps(\tau)^2 h$,
  where
  \begin{equation*}
    \rho_\eps(\tau)=
    \begin{cases}
      \eps          & \text{if $|\tau|\le L/2$}\\
      |\tau|-L/2+\eps  & \text{if $|\tau|\ge  L/2$.}
    \end{cases}
 \end{equation*}
\end{itemize}
We denote by $M_\eps$ the new Riemannian manifold. We are interested
in studying the limit, as $\eps \to 0$, of the spectrum
$\lambda_k^p(\eps)=\lambda_k^p(g_\eps)$, $k\geq 1$, of the Hodge-de
Rham operator acting on $p$-forms defined on the manifold $M_\eps$. We
remark that $g_\eps$ is only continuous. 

The Hodge-de Rham operator is defined in this case as follows
(see~\cite{AC2} for more details).  The manifold is the union of two
smooth parts with boundary. For a manifold $M=M_1\cup M_2$, denote by
$D_1,D_2$ the Gau\ss-Bonnet operator on each part. The quadratic form
$q(\phi)=\int_{M_1}|D_1(\phi{\restriction}_{M_1})|^2 +
\int_{M_2}|D_2(\phi{\restriction}_{M_2})|^2$ is well defined and
closed on the domain
  \begin{equation*}
    \dom(q) =\bigset{\phi=(\phi_1,\phi_2) \in \Sob{M_1} \times \Sob{M_2}}
                { \phi_1{\restriction}_{\partial M_1} 
                    \stackrel{\Lsymb_2} = 
                  \phi_2{\restriction}_{\partial M_2} },
  \end{equation*}
  and on this space the total Gau\ss-Bonnet operator is defined and
  selfadjoint.  The Hodge-de Rham operator of $M$ is then defined as
  the operator obtained by the polarization of the quadratic form $q$.
  This gives compatibility conditions between $\phi_1$ and $\phi_2$ on
  the commun boundary, these conditions are explained in detail in the
  next section.

  Finally we remark that it is not a loss of generality to concider
  only continuous metrics: the family $(g_\eps)_{\eps>0}$ can be
  approched by a family of smooth Riemannian metrics
  $(g_{\eps,\eta})_{\eps>0}$ such that
  \begin{equation}
    \label{eq:met.est}
    \e^{-\eta} g_\eps \le g_{\eps,\eta}\le \e^\eta g_\eps
  \end{equation}
   for all $\eps$.
\begin{proof}
  Let $f_\eta$ be a smooth, increasing function on $\R^+$ such that
 \begin{equation*}
    f_\eta(r)=1 \qquad \text{for $r \leq 1$} 
        \qquad \text{and} \qquad 
    f_\eta(r)=r \qquad \text{for $r\geq 1+\eta$}.
 \end{equation*}
 Then the metric $g_{\eps,\eta}=d\tau^2+f_{\eps,\eta}(\tau)^2 h$ on
 $\mcM_\eps$ with
 $f_{\eps,\eta}$ defined by
 \begin{equation*}
    f_{\eps,\eta}(\tau)=
     \begin{cases}
       \eps & \text{if $|\tau|\leq L/2$}, \\
       \eps f_\eta \Bigl(\frac{|\tau|-L/2+\eps}{\eps} \Bigr)
            & \text{if $|\tau|\geq L/2$}
   \end{cases}
 \end{equation*}
 satisfies the estimate~\eqref{eq:met.est}.
\end{proof}
A result of Dodziuk~\cite[Prop.~3.3]{D} implies now, that the
corresponding eigenvalues satisfy
\begin{equation*}
  \e^{-(n+2p)\eta} \lambda_k^p(g_\eps) \le 
  \lambda_k^p(g_{\eps,\eta}) \le 
  \e^{(n+2p)\eta} \lambda_k^p(g_\eps).
\end{equation*}
Note that the result of Dodziuk also applies to our singular metrics,
based on the Hodge decomposition and the fact that the spectrum away
from $0$ is given by exact forms.  Hence, it is enough to prove our
results only for a family of continuous (but piecewise smooth)
metrics, and the convergence results extend also to a family of smooth
metrics. The above definition of a family of non-smooth metrics will
simplify some of our arguments in the next section. Namely, we can
solve certain differential equations explicitly due to the special
form of the metric on the cones and the handle.

\section{Description of the Hodge-de Rham operator on $\mcM_\eps$}
\label{sec:lap}

In this section we express the norm of a $p$-form, the Gau\ss -Bonnet,
the Hodge-de Rham operator and its associated quadratic form in the
new coordinates.  On the cones $\mcC_\eps^\pm$, we use the same
parametrisation of the forms as the one introduced in \cite{BS} and
\cite{BMS}, namely a $p$-form $\phi$ can be written as
\begin{equation*}
  \phi = 
  dt \wedge t^{-(n/2-p+1)}\beta_\pm + 
              t^{-(n/2-p)}\alpha_\pm
\end{equation*}
and we set
\begin{equation*}
  U_\pm \phi := 
  \sigma_\pm := (\beta_\pm,\alpha_\pm)
  \in \Ci { {]}\eps, 1{[}, 
        \Ci{\Lambda^{p-1}T^*\Sigma} \oplus \Ci{\Lambda^p T^*\Sigma}}.
\end{equation*}

Similarly, on the handle, we have
\begin{equation*}
 \phi =
 dr \wedge \eps^{-(n/2-p+1)} \beta +
              \eps^{-(n/2-p)}\alpha 
\end{equation*}
and we set
\begin{equation*}
  U \phi :=\sigma:=(\beta,\alpha)
  \in \Ci { {]}0,L{[}, 
        \Ci{\Lambda^{p-1}T^*\Sigma} \oplus \Ci{\Lambda^p T^*\Sigma}}.
\end{equation*}
Since we included the factor $\rho_\eps$ of the (warped) product
$g_\eps = dt^2 + \rho_\eps(t)^2 h$ in the definition of the
transformation, it is now straightforward to see that $U_\pm$ extends
to a unitary operator on the corresponding $\Lsymb[2]$-spaces
and similarly for $U$. In particular, we have
\begin{multline}
  \label{eq:L2conser}
  \| \phi \|^2_{\Lsqr{\Lambda^\bullet T^* \mcM_\eps, g_\eps}} = 
  \int_{\mcM_\eps} |\phi|_{g_\eps}^2 d\vol_{g_\eps} \\= 
  \sum_{s=\pm} \int_\eps^1 \bigl[ 
        |\beta_s(t)|^2 + |\alpha_s(t)|^2 \bigr] dt
    + \int_0^L \bigl[ |\beta(t)|^2 + |\alpha(t)|^2 \bigr] dt,
\end{multline}
where $|\cdot|$ denotes the $\Lsymb[2]$-norm on $\Lsqr{\Lambda^\bullet
  T^* \Sigma,h}$.

We can now transform the \emph{Gau\ss -Bonnet operator} $D:=d+d^\ast$,
which in fact depends on $\eps$ as the metric does, using the
transformations $U_\pm$ resp.~$U$ and obtain
\begin{align*}
  UDU^*=&\begin{pmatrix} 0&1\\
               -1&0   
        \end{pmatrix}
       \Big( \partial_t+\frac 1 \eps 
                 \begin{pmatrix} 0&-D_0\\
                        -D_0&0\end{pmatrix} \Big)& \text{on the handle
                        $\mcA_\eps$ and}\\
  U_\pm DU_\pm^*=&\begin{pmatrix} 0&1\\
               -1&0   
        \end{pmatrix}
       \Big( \partial_t+\frac 1 t
                 \begin{pmatrix} \dfrac n 2 -P&-D_0\\
                -D_0 & P-\dfrac n 2 \end{pmatrix} \Big)
        &\text{on the cones $\mcC_\eps^\pm$,}
\end{align*}
where $D_0=d_0+d_0^\ast$ denotes the Gau\ss -Bonnet operator on the
compact Riemannian manifold $(\Sigma,h)$ and where $P$ is the linear
operator multiplying with the degree of the form.  For further
purposes, it will be useful to denote
\begin{equation*}\label{opeA}
  A_0 = \begin{pmatrix} 0&-D_0\\
                -D_0&0\end{pmatrix} \qquad\text{and}\qquad
  A   = \begin{pmatrix} \dfrac n 2 -P&-D_0\\
                -D_0 & P - \dfrac{n}{2} \end{pmatrix}
\end{equation*}
the parts, in the transformed Gau\ss -Bonnet operators, acting
non-trivially in the transversal direction $\Sigma$.

In this representation, a piecewise smooth form $\phi$ is in the
domain of $D$ if and only if $\phi$ is in the $\Sobsymb[1]$-space of
each part and if the components of $U\phi$ and $U_\pm\phi$ satisfy
the \emph{compatibility} or \emph{transmission conditions}
\begin{equation}
  \label{recolfq}
  \begin{cases}
    \beta(L)=\beta_+(\eps),& \quad \beta(0)=-\beta_-(\eps)\\
    \alpha(L)=\alpha_+(\eps),& \quad \alpha(0)=\alpha_-(\eps).
  \end{cases}
\end{equation}
The \emph{Hodge-de Rham operator} is now given by $D^2$. A simple
calculation shows that on the handle $\mcA_\eps$, we have
\begin{equation*}
  UD^2U^* = -\partial_t^2+\frac{1}{\eps^2}A_0^2
  =-\partial_t^2+\frac{1}{\eps^2}\begin{pmatrix} \Delta_\Sigma&0\\0&
 \Delta_\Sigma\end{pmatrix},\end{equation*}
 where $\Delta_\Sigma=D_0^2$ denotes the Hodge-de
Rham operator of the Riemannian manifold $(\Sigma,h)$. Similarly, on
the cones $\mcC_\eps^\pm$ we have the expression
\begin{equation*}
  U_\pm D^2U_\pm^* = -\partial_t^2+\frac{1}{t^2}(A+A^2),
\end{equation*}
where
\begin{equation}
  \label{opcone}
  A + A^2 =
  \begin{pmatrix} 
     \Delta_\Sigma + \Bigl(\dfrac{n}{2}-P\Bigr) \Bigl(\dfrac{n}{2}-P+1\Bigr)
          &  -2d_0^\ast\\
     -2d_0&\Delta_\Sigma + \Bigl(\dfrac{n}{2}-P \Bigr)
                                   \Bigl(\dfrac{n}{2}-P-1\Bigr)
   \end{pmatrix}.
\end{equation}
The domain of $D^2$ consists of those forms $\phi$ in the domain of
$D$ such that $D\phi$ is also in the domain of $D$. In particular, the
domain of the transformed operator $UD^2U^*$ consists of pairs of
forms satisfying --- in addition to~\eqref{recolfq} --- the following
\emph{compatibility} or \emph{transmission conditions} (of first
order) on the derivatives:
\begin{subequations}
  \label{recolop}
  \begin{align}
      \beta'(L) &= \beta_+'(\eps) +\dfrac{1}{\eps} \Bigl(\dfrac{n}{2}-P
      \Bigr) \beta_+(\eps), &  \beta'(0) &=
      \beta_-'(\eps)+\dfrac{1}{\eps}
      \Bigl(\dfrac{n}{2}-P \Bigr)\beta_-(\eps)\\
      \alpha'(L) &= \alpha_+'(\eps)- \frac{1}{\eps} \Bigl(
      \dfrac{n}{2}-P \Bigr)\alpha_+(\eps), &  \alpha'(0) &=
      -\alpha_-'(\eps)+\dfrac{1}{\eps} \Bigl(\dfrac{n}{2}-P
      \Bigr)\alpha_-(\eps).
  \end{align}
\end{subequations}
Let us now compute the expression
\begin{equation*}
  \| D\phi \|^2_{\Lsqr{\Lambda^\bullet T^*\mcM_\eps,g_\eps}} 
  = \int_{\mcM_\eps} |D\phi|_{g_\eps}^2 d \vol_{g_\eps},
\end{equation*}
i.e., the quadratic form on $\mcM_\eps$, for a form $\phi$ in the
domain of $D$, and with support in $\mcM_\eps$ in terms of
\begin{equation*}
  \sigma_\pm = \begin{pmatrix}\beta_\pm\\ \alpha_\pm \end{pmatrix} 
             = U_\pm \phi \qquad \text{and} \qquad
  \sigma     = \begin{pmatrix}\beta\\ \alpha \end{pmatrix}=U\phi
\end{equation*}
using the isometries $U_\pm$ and $U$.

Denote by $\iprod \cdot \cdot$ the scalar product in $\Lsqr
{\Lambda^\bullet T^*\Sigma,h} \oplus \Lsqr {\Lambda^\bullet
  T^*\Sigma,h}$ and by $\mcC_\eps$ one of the two cones, oriented by
$dt \wedge d\vol_\Sigma$. The expression of the transformed quadratic
form on the cone is then
\begin{align*}
  \int_{\mcC_\eps} |D\phi|_{g_\eps}^2 d \vol_{g_\eps}
    &= \int_\eps^1 
     \left| \Bigl(\partial_t + \frac{1}{t} A \Bigr) 
            \sigma_\pm\right|^2 dt\\
    &= \int_\eps^1
     \Bigl[\, |\sigma_\pm'|^2+\frac{2}{t} \iprod {\sigma_\pm'} {A\sigma_\pm}
             +\frac{1}{t^2} |A\sigma_\pm|^2\,\Bigr] dt\\
    &= \int_\eps^1
     \Bigl[\, |\sigma_\pm'|^2+\partial_t 
           \Bigl( \frac{1}{t} \iprod {\sigma_\pm} {A\sigma_\pm} \Bigr)
       + \frac{1}{t^2} \bigl( \iprod {\sigma_\pm} {A\sigma_\pm} 
       + |A\sigma_\pm|^2 \bigr)\,\Bigr] dt\\
    &= \int_\eps^1
     \Bigl[\, |\sigma_\pm'|^2+\frac{1}{t^2}
              \iprod{\sigma_\pm} {(A+A^2)\sigma_\pm} \,\Bigr] dt
       - \frac{1}{\eps} \bigiprod {\sigma_\pm(\eps)}{A\sigma_\pm(\eps)}.
\intertext{Similarly, on the handle we have}
  \int_{\mcA_\eps} |D \phi|_{g_\eps}^2 d \vol_{g_\eps}
    &= \int_0^L 
     \left| \Bigl(\partial_t + \frac{1}{\eps} A_0 \Bigr)
            \sigma \right|^2 dt \\
    &= \int_0^L
     \Bigl[\, |\sigma'|^2 + \frac{1}{\eps^2} |A_0\sigma|^2 +
        \frac{2}{\eps} \iprod {\sigma'} {A_0\sigma} \,\Bigr]dt\\
    &= \int_0^L
     \Bigl[\, |\sigma'|^2 + \frac{1}{\eps^2} |A_0\sigma|^2 \,\Bigr]dt\\
      &\hspace{3cm} 
         + \frac{1}{\eps}
             \Bigl( \bigiprod{\sigma(L)} {A_0\sigma(L)} -
                    \bigiprod{\sigma(0)} {A_0\sigma(0)} \Bigl).
\end{align*}
The total boundary term is
\begin{multline*}
  \qf b(\phi,\phi) =
  \Bigl(- \bigiprod {\sigma_+(\eps)} {A \sigma_+(\eps)} -
          \bigiprod {\sigma_-(\eps)} {A \sigma_-(\eps)} \\
        + \bigiprod {\sigma(L)} {A_0 \sigma(L)} -
          \bigiprod {\sigma(0)} {A_0 \sigma(0)}\Bigr).
\end{multline*}
Using the compatibility conditions~\eqref{recolfq} and the relation
\begin{equation*}
  A = A_0 - \begin{pmatrix} P-\frac{n}{2}&0\\
                0&\frac{n}{2}-P\end{pmatrix},
\end{equation*}
we obtain for the boundary term $\qf b(\phi,\phi)$ the following expression
\begin{equation*}
  \qf b (\phi,\phi) 
   =\sum_{s=\pm} \Bigiprod {\sigma_s(\eps)} 
         {\begin{pmatrix} P-\frac{n}{2}&0\\
                      0&\frac{n}{2}-P\end{pmatrix}\sigma_s(\eps)},
\end{equation*}
which does not contain derivatives any more.  Finally, we can express
the quadratic form associated to the Hodge-de Rham operator on
$\mcM_\eps$ as
\begin{multline}
  \label{locform}
  \int_{\mcM_\eps} |D\phi|^2 d\vol_{g_\eps} = \sum_{s=\pm}
  \int_\eps^1 \Bigl( |\sigma'_s|^2 +\frac{1}{t^2}
  \iprod {\sigma_s} {(A+A^2) \sigma_s} \Bigr) dt\\
+ \int_0^L
    \Bigl( |\sigma'|^2 + \frac{1}{\eps^2}
           |A_0 \sigma |^2 \Bigr) dt  + \frac{1}{\eps} \qf b(\phi,\phi)
\end{multline}
for $p-$forms supported in $\mcM_\eps$.

\section{Asymptotic estimates}
\label{sec:asymp}
\subsection{Spectrum of the operator $A+A^2$}\label{specA}
The expression of $A+A^2$ was given in formula~\eqref{opcone}.  We
remark that the function
\begin{equation*}
  f(p)= \Bigl( \frac n 2 - p \Bigr) \Bigl( \frac n 2 - p - 1 \Bigr)
\end{equation*}
has zeros for the
values $n/2$ and $n/2-1$. In particular, for $p \in \N$ we have always
$f(p) \geq 0$ if $n$ is even and $f(p) \geq - 1/4$ if $n$ is odd. The
value $-1/4$ is obtained only for $p=(n-1)/2$. Setting
\begin{equation*}
  a_p:= \frac{n+1}{2}-p
\end{equation*}
we have the relation $f(p)=a_{p+1}^2- 1/4$.

The following lemma is a direct consequence of the Hodge decomposition
theorem for the compact manifold $(\Sigma,h)$ and the expression of
the operator $A+A^2$ on each of the subspaces given in the lemma:
\begin{lem}
  \label{decompos}
  The space $\Lsqr {\Lambda^{p-1}\Sigma} \oplus \Lsqr {\Lambda^p
    \Sigma}$ is the orthonormal sum of the following five spaces, and
  $A+A^2$ acts on these spaces as indicated:
  \begin{align*}
    \mcH_1 &= \{(\beta,0);\,\Delta_\Sigma \beta =0\,\},
        & (A+A^2)(\beta,0)  &= (f(p-2)\beta,0),\\
    \mcH_2 &=\{(0,\alpha);\,\Delta_\Sigma \alpha =0\}, 
        & (A+A^2)(0,\alpha) &= (0,f(p)\alpha),\\
     \mcH_3&=\{(\beta,0);\, \text{$\beta$ exact}\,\}, 
        & (A+A^2)(\beta,0)  &= 
             \bigl((\Delta_\Sigma+f(p-2))\beta,0\bigr),\\
     \mcH_4&=\{(0,\alpha);\, \text{$\alpha$ co-exact}\,\}, 
        & (A+A^2)(0,\alpha) &=
            \bigl(0,(\Delta_\Sigma+f(p))\alpha\bigr),\\
     \mcH_5&=\{(\beta,\alpha);\, 
          \text{$\beta$ co-exact, $\alpha$ exact} \, \},
        \hspace*{-1ex}&  (A+A^2)(\beta,\alpha) &=\\
      &&&  \hspace*{-25ex} 
         = \bigl((\Delta_\Sigma+f(p-2))\beta - 2d_0^\ast\alpha,
                  (\Delta_\Sigma+f(p))\alpha-2d_0\beta \bigr).
  \end{align*}
  In addition, this decomposition is preserved by $A+A^2$ and $A_0^2$.
\end{lem}

We can now compute explicitly the eigenvalues of the operator $A^2+A$
in terms of the spectrum of the Hodge-de Rham operator
on $\Sigma$.  Clearly, on the spaces $\mcH_i$, $i=1,\dots,4$, the
operator $A^2+A$ is already diagonalised provided $\alpha$ and $\beta$
are eigenforms of $\Delta_\Sigma$.

If $(\beta,\alpha) \in \mcH_5$ is an eigenvector of $A+A^2$ for the
eigenvalue $\lambda$ then they satisfy the equations
\begin{align}
  \label{eq1}
  \bigl( \Delta_\Sigma + f(p-2) - \lambda \bigr) \beta 
   &= 2 d_0^\ast \alpha\\
  \label{eq2}
  \bigl( \Delta_\Sigma + f(p)   - \lambda \bigr) \alpha
   &= 2d_0\beta.
\end{align}
Applying $d_0$ to the first and
$\bigl(\Delta_\Sigma+f(p-2)-\lambda\bigr)$ to the second equation, and
substituting the $\beta$ term, leads to the equation
\begin{equation}
  \label{eq3}
  4\Delta_\Sigma\alpha 
  = \bigl( \Delta_\Sigma + f(p-2) - \lambda \bigr)
    \bigl( \Delta_\Sigma + f(p)   - \lambda \bigr) \alpha
\end{equation}
for $\alpha$. If $\alpha$ is an exact eigenform with
$\Delta_\Sigma\alpha=\mu^2\alpha$ then $\lambda$ is a solution of the
second order polynomial equation
\begin{equation}
  \label{eq4}
  4\mu^2 = \bigl( \mu^2 + f(p-2) - \lambda \bigr) 
           \bigl( \mu^2 + f(p)   - \lambda \bigr).
\end{equation}
A direct computation shows that the solutions of this equation
are
\begin{equation*}
  \lambda_\pm(\mu^2)=\gamma_\pm(\mu^2)(\gamma_\pm(\mu^2)+1)
\end{equation*}
where
\begin{equation}
  \label{gam5}
  \gamma_\pm(\mu^2) = -\frac12 +\Bigl|\sqrt{\mu^2+a_p^2}\pm 1\Bigr|.
\end{equation}

Now if $(\alpha_k)_{k\in \N}$ is a complete family of exact eigenforms
with corresponding eigenvalues $(\mu^2_k)_{k\in \N}$, then it is easily
seen that the family
\begin{equation}
  \label{h5.ev}
  \Bigl( \frac{2} {{\mu_k}^2+f(p-2)-\lambda_s(\mu_k^2)} d^\ast_0\alpha_k,
     \alpha_k \Bigr)_{k\in \N, s=\pm}
\end{equation}
defines a basis of $\mcH_5$ of eigenvectors of $A+A^2$ with
corresponding eigenvalues $\{\lambda_s(\mu_k^2)\}_{k\in \N,s=\pm}$.

For further purpose, it will be very convenient to write the
eigenvalues of $A+A^2$ in the form $\gamma(\gamma+1)$, as we have
already done in the above calculation of the spectrum of $A+A^2$ on
$\mcH_5$.  The spectrum of the restriction of $A+A^2$ on $\mcH_3$ is
given by $\gamma(\mu^2)(\gamma(\mu^2)+1)$, where
\begin{equation}\begin{split}
  \label{gam3}
  \gamma(\mu^2) 
 & = - \frac12 + \sqrt{\mu^2+\Bigl(\frac{n+3}{2}-p\Bigr)^2}\\
  &= - \frac12 + \sqrt{\mu^2+\Bigl(a_p+1\Bigr)^2}\\
  &= - \frac12 + \sqrt{\mu^2+a_{p-1}^2}
\end{split}\end{equation}
for $\mu^2$ running over the exact spectrum of $\Delta_\Sigma$ acting
on $(p-1)$-forms.

Similarly, the spectrum of $A+A^2$ restricted to $\mcH_4$ is given by
$\gamma(\mu^2)(\gamma(\mu^2)+1)$, where
\begin{equation}\begin{split}
  \label{gam4}
  \gamma(\mu^2)
  &= - \frac12 + \sqrt{\mu^2+\Bigl(\frac{n-1}{2}-p\Bigr)^2}\\
  &= - \frac12 + \sqrt{\mu^2+\Bigl(a_p-1\Bigr)^2} \\
  &= - \frac12 + \sqrt{\mu^2+a_{p+1}^2}
\end{split}\end{equation} 
for $\mu^2$ running over the co-exact spectrum of $\Delta_\Sigma$
acting on $p$-forms.

The spectrum of $A+A^2$ on $\mcH_1$, is $\gamma(\gamma+1)$ with multiplicity $b_{p-1}(\Sigma)$
where \begin{equation}
  \label{gam1}
  \gamma=-\frac12+\left|a_{p-1}\right|=-\frac12+\left|\frac{n+1}{2}-p+1\right|.
\end{equation} 
The spectrum of $A+A^2$ on $\mcH_2$, is $\gamma(\gamma+1)$ with multiplicity $b_{p}(\Sigma)$
where \begin{equation}
  \label{gam2}
  \gamma=-\frac12+\left|a_{p+1}\right|=-\frac12+\left|\frac{n+1}{2}-p-1\right|.
\end{equation}
\begin{rem}
  The decomposition given in \lref{decompos} is also preserved by
  $A_0^2$. Therefore, the expression~\eqref{locform} of the quadratic
  form for a $p$-form supported in $\mcM_\eps$ shows that the
  pointwise decomposition of a form is preserved by the quadratic
  form.  Namely, if $\phi=\sum_{1 \leq i\leq 5}\phi^i$ with
  $\phi^i(t)\in \mcH_i$ for all $t$, then
  \begin{equation*}\label{decomposition}
    \int|D \phi|^2 = \sum_{1\leq i\leq 5}\int|D \phi^i|^2.
  \end{equation*}

  For our asymptotic analysis below we need a spectral decomposition
  in a low and high eigenvalue part. Namely, we need the decomposition
  \begin{equation}
    \label{eq:decomp2}
    \phi^3+\phi^4+\phi^5=\phi_\Lambda+\phi^\Lambda,
  \end{equation}
  where $U_\pm \phi_\Lambda$ and $U \phi_\Lambda$ belong (pointwise)
  to the orthogonal sum of the eigenspace of $A_0^2$ associated to the
  eigenvalues smaller that $\Lambda^2$. Similarly, $U_\pm
  \phi^\Lambda$ and $U \phi^\Lambda$ belong (pointwise) to the
  orthogonal sum of the ei\-gen\-space of $A_0^2$ associated to the
  eigenvalues strictly larger than $\Lambda^2$.
\end{rem}

\subsection{Study of a sequence of eigenforms}
\label{suite}
We consider now a sequence $\eps_m$ converging to $0$ such that there
is a sequence $\lambda_m$ of eigenvalues of the Hodge-de Rham operator
$\Delta_{\eps_m}$ on $M_{\eps_m}$ and converging to some $\lambda$. Let 
$\phi_m$ be the corresponding normalised $p$-eigenform. In the following we 
will write $\eps=\eps_m.$ Thus
\begin{equation*}
  \Delta_\eps \phi_m=\lambda_m \phi_m, \qquad \|\phi_m \|=1.
\end{equation*}
We want to understand the behaviour of $\phi_m$ when $m\to\infty$.
Since this sequence is bounded in $\Sobsymb[1]_{\mathrm{loc}}$ and by
elliptic regularity, after passing to a subsequence, we can assume
that $\phi_m$ converges to $\phi$ on $M\setminus \mcV $ in the
$\Sobsymb[1]$-topology and also in $\Contsymb[\infty]_{\mathrm{loc}}$.
Similarly, we can also assume that $\phi_m$ converge to $\phi$ on each
of the cones $\mcC^\pm_\eta$ for fixed $\eta>0$ such that $\eps_m \le
\eta$.  The main difficulty is to understand the behaviour of $\phi_m$
on $\mcM_\eps$. For this purpose we introduce a smooth cut-off
function $\chi$ with support in $\mcM_\eps,\, 0\le \chi\le 1,$ and
such that $\chi=1$ on $\mcA_\eps\cup (\mcC^+_\eps\setminus
\mcC^+_{1/2})\cup(\mcC^-_\eps\setminus \mcC^-_{1/2}).$

On $\mcM_\eps$ we have the decomposition
\begin{equation}\label{decomp}
  \phi_m = \phi_m^1+\phi_m^2+\phi_{m,\Lambda}+\phi_m^\Lambda
\end{equation}
of \lref{decompos} and~\eqref{eq:decomp2}.

\subsubsection{Non-harmonic terms}
We study here the behaviour of the last two terms.
\begin{lem}
  \label{lem:high.ev}
  The high-energy component of the eigenforms can be
  estimated near the handle by
  \begin{equation*}
    \|\phi_m^\Lambda\|_{\Lsqr {\mcA_\eps}}^2
      \le C \frac{\eps^2}{\Lambda^2}
      \qquad \text{and} \qquad
    \|\phi_m^\Lambda\|^2_{\Lsqr {\mcC^\pm_\eps\setminus \mcC^\pm_{\eta}}}
      \le C\frac{\eta^2}{\Lambda^2}
  \end{equation*}
  provided $\Lambda$ is large enough and $\eps=\eps_m \le \eta$. The
  estimate is uniform in $m \to \infty$.
\end{lem}
\begin{proof}
Let $\sigma_m=U \chi\phi^\Lambda_m$ and
  $\sigma_{\pm,m}=U_\pm \chi\phi^\Lambda_m,$ we have that
  \begin{align*}
\|D\chi\phi^\Lambda_m\|^2_{\Lsymb[2]}&=
\int_{\mcM_\eps}|d\chi|^2|\phi^\Lambda_m|^2|+\int_{\mcM_\eps}\chi^2
|D\phi^\Lambda_m|^2\\
&\leq |d\chi|_\infty^2+|\chi|_\infty^2\int_{\mcM_\eps}|D\phi^\Lambda_m|^2
\leq |d\chi|_\infty^2+|\chi|_\infty^2 \lambda_m=C_\chi(\lambda_m)  
\end{align*}
is uniformly bounded, but on the other hand
  \begin{multline}  
         \|D\chi\phi^\Lambda_m\|^2_{\Lsymb[2]}\ge
      \sum_{s=\pm} \int_\eps^{1/2} 
       \Bigl[ |\sigma_{\pm,m}'|^2+ 
         \frac{\iprod {(A^2+A)\sigma_{\pm,m}} {\sigma_{\pm,m}}} 
              {t^2} \Bigr] dt\\
    \label{eq:est1}
      + \int_0^L \Bigl[ |\sigma_m'|^2 +
         \frac{|A_0\sigma_m|^2} {\eps^2} \Bigr] dt
       - \frac{n+1}{2}\bigl[ |\sigma_m(0)|^2+\sigma_m(L)|^2\bigr].
  \end{multline}
  The boundary term can be estimated using the following optimal
  inequality
  \begin{equation*}
    \int_0^L \Bigl[ |v'(t)|^2 + 
                \frac{\Lambda^2}{\eps^2}|v(t)|^2 \Bigr]dt 
    \ge \frac {\Lambda}{\eps} 
               \tanh \Bigl(\frac{\Lambda L}{2\eps}\Bigr)
               \bigl[|v(0)|^2+|v(L)|^2\bigr],
  \end{equation*}
  which is true for all $v \in \Sob{[0,L]}$. Namely, if we choose
  $\Lambda>0$ sufficiently large such that
  \begin{equation*}
    \frac{\Lambda}{\eps} \tanh \Bigl(\frac{\Lambda L}{2\eps} \Bigr)
    \ge (n+1)
  \end{equation*}
  for all $\eps \in {]}0,1]$, then the boundary term can be estimated
  in terms of the last integral in~\eqref{eq:est1}.  In addition, the
  spectrum of the restriction of the operator $A^2+A$ to the
  orthogonal sum of the eigenspaces of $A_0^2$ associated to the
  eigenvalues strictly larger that $\Lambda^2$, is bounded from below
  by $\Lambda^2/2$. Consequently, we obtain
  \begin{align*}
    C_\chi(\lambda_m) &
      \ge \sum_{s=\pm} 
           \int_\eps^{1/2} \left[ |\sigma_{\pm,m}'|^2 + 
                \frac{\Lambda^2|\sigma_{\pm,m}|^2}{2 t^2}\right]dt
         +\frac12  \int_0^L 
           \left[ |\sigma_m'|^2+\frac{|A_ 0\sigma_m|^2}{\eps^2}\right]dt\\
      &\ge \sum_{s=\pm} \int_\eps^{1/2}
      \frac{\Lambda^2|\sigma_{\pm,m}|^2}{2 t^2}dt+
      \int_0^L \frac{\Lambda^2|\sigma_m|^2}{2\eps^2}dt\\
      &\ge \sum_{s=\pm} \int_\eps^\eta
      \frac{\Lambda^2|\sigma_{\pm,m}|^2}{2 \eta^2}dt+
      \int_0^L \frac{\Lambda^2|\sigma_m|^2}{2\eps^2}dt \\
  \end{align*}
  and we are done with $C=4C_\chi(\lambda)$ since $\lambda_m \to \lambda$.
\end{proof}

The next lemma says that also the non-harmonic low energy part of
$\phi_m$ goes to zero on the handle when $m\to\infty$.
\begin{lem}
  \label{lem:low.ev}
  \begin{equation*}
    \lim_{m\to\infty} \|\phi_{m,\Lambda}\|^2_{\Lsqr{\mcA_\eps}}=0.
  \end{equation*}
\end{lem}
\begin{proof}
  Let $U_s\chi\phi_m=\sigma_{s,m}$ for $s=\emptyset,+,-$.  Since
  $\phi_{m,\Lambda}$ is a finite sum of forms, which transversally are
  (non-harmonic) eigenforms of $A^2+A$ and $ A^2_0$, we can assume
  that $A_0^2\sigma_{s,m}=\mu^2\sigma_{s,m}$ with $\mu\neq 0$, and
  $(A^2+A)\sigma_{s,m}=\gamma(\gamma+1)\sigma_{s,m}$ where $\gamma\ge
  -1/2$ depends on $\mu$ as in~\eqref{gam5},~\eqref{gam3}
  and~\eqref{gam4}.

  On the handle, i.e., on $[0,L]$, $\sigma_m$ satisfies the
  (form-valued) equation
  \begin{equation*}
    -\sigma_m''(t)+\frac{\mu^2}{\eps^2}\sigma_m(t)=\lambda_m \sigma_m(t).
  \end{equation*}
  Consequently, if $\delta_m=\sqrt{\mu^2/\eps^2-\lambda_m}$, we can
  write
  \begin{equation*}
    \sigma_m(t) 
    = \frac{1}{\sqrt{\eps}}
      \bigl[ a_m \e^{-\delta_m t}+b_m \e^{-\delta_m (L-t)} \bigr],
  \end{equation*}
  where the coefficients $a_m$, $b_m$ are pairs of forms on
  $\Sigma$.
  
  It is not hard to check that there is a constant $C$
  independent of $m$ such that
  \begin{equation*}
    C^{-1} (|a_m|^2+ |b_m|^2)
    \le \int_0^L  |\sigma_m(t)|^2 dt
    \le C(|a_m|^2 + |b_m|^2)
  \end{equation*}
  where $|\cdot|$ denotes the $\Lsymb[2]$-norm of pairs of forms on
  $\Sigma$.

  Our aim is to show that $a_m$ and $b_m$ converge to $0$. To do so,
  we need also the behaviour of solutions on the cones.  Namely, on
  $[\eps,1/2]$, the transformed eigenform $\sigma_{\pm,m}$ solves the
  equation
  \begin{equation*}
    -\sigma_{\pm,m}''(t) + \frac{\gamma(\gamma+1)}{t^2}\sigma_{\pm,m}(t)
    =\lambda_m\sigma_{\pm,m}.
  \end{equation*}
  Hence we can express the solution of the equation in terms of
  Bessel's functions. As a result, there are entire functions
  $F_\gamma$ and $G_\gamma$ with $F_\gamma(0)=G_\gamma(0)=1$, such
  that the solutions are linear combinations of
\begin{align*}
    f_\gamma(t)&=t^{\gamma+1} F_\gamma(\lambda_m t^2)
\\   g_\gamma(t)&=
  \begin{cases}
     t^{-\gamma} G_\gamma(\lambda_m t^2) & 
                \text{if $\gamma +1/2 \notin \N$}\\
     t^{-\gamma} G_\gamma(\lambda_m t^2) + a\log(t) f_\gamma(t) &
                \text{if $\gamma +1/2 \in \N$.}
  \end{cases}
\end{align*}
Namely, there exist pairs of forms $c_{\pm,m}$, $d_{\pm,m}$ on
$\Sigma$ (independent of $t$), such that
\begin{equation}\label{bessel}
    \sigma_{\pm,m}(t)
    =  c_{\pm,m}f_\gamma(t)
     + d_{\pm,m}g_\gamma(t)
  \end{equation}
    
  In both cases, we obtain the estimate
  \begin{multline}
  \label{growth}
        C^{-1} \bigl( |c_{\pm,m}|^2 + h_\gamma(\eps) |d_{\pm,m}|^2\bigr)
    \le \int_\eps^{1/2} |\sigma_{\pm,m}(t)|^2 dt \\
    \le C      \bigl( |c_{\pm,m}|^2 + h_\gamma(\eps) |d_{\pm,m}|^2\bigr)
  \end{multline}
  where
  \begin{equation*}
    h_\gamma(\eps) \sim
   \begin{cases}
      \eps^{-2\gamma+1} &\text{if $\gamma > 1/2$}\\
      |\log(\eps)|      &\text{if $\gamma = 1/2$}\\
      1                 &\text{if $\gamma < 1/2$.}
    \end{cases}
  \end{equation*}

  We now use the transmission conditions~\eqref{recolfq} to combine
  the solutions on the handle and the cones. Let
  \begin{equation}
    \label{eq:def.j}
    J = \begin{pmatrix} -\id & 0\\ 0 & \id \end{pmatrix}.
  \end{equation}
  Then the transmission conditions~\eqref{recolfq} read as
  \begin{equation}
    \label{transi1}
    \frac{1}{\sqrt{\eps}} \bigl[a_m \e^{-\delta_m L}+b_m \bigr]
     = \sigma_{+,m}(\eps)
  \end{equation}
  and
  \begin{equation}
    \label{transi2}
    \frac{1}{\sqrt{\eps}} \bigl[a_m +b_m \e^{-\delta_m L} \bigr]
    = J \sigma_{-,m}(\eps).
  \end{equation}
  Since $a_m$ and $b_m$ belong to a compact set (namely, to a ball of
  the finite-dimen\-sio\-nal space of eigenforms of $A_0^2$ below
  $\Lambda^2$), we can assume (after passing to a subsequence) that
  \begin{equation*}
    \lim_{m\to\infty} a_m = a_\infty
    \qquad \text{and} \qquad
    \lim_{m\to\infty} b_m = b_\infty.
  \end{equation*}
  Recall that the main point is to show that $a_\infty=b_\infty=0$.

  The transmission conditions~\eqref{transi1} and~\eqref{transi2}, the
  behaviour of $\delta_m \sim \mu/\eps$ as $\eps$ goes to 0, and the
  fact that the sequence $c_{\pm,m}$ is bounded (cf.~\eqref{growth}),
  imply that
  \begin{equation}
    \label{b1}
    \lim_{m\to\infty} d_{+,m}\eps^{-\gamma+1/2}=b_ \infty
  \end{equation}
  and
  \begin{equation}
    \label{b2} 
    \lim_{m\to\infty} d_{-,m}\eps^{-\gamma+1/2}=Ja_ \infty.
  \end{equation}
  We conclude from these last equalities together with~\eqref{growth}
  that $a_\infty=b_\infty=0$ in the case $\gamma < \frac12$. A similar
  argument holds in the case $\gamma=\frac12$.

  It remains to consider the case $\gamma>\frac12$. From the transmission
  condition of first order~\eqref{recolop} we obtain
  \begin{equation}
    \label{transi1'}
    \frac{\delta_m}{\sqrt\eps} \bigl[ -a_m \e^{-\delta_m L}+b_m \bigr]
    = \sigma'_{+,m}(\eps)-\frac{1}{\eps} 
                    \Bigl( \frac{n}{2}-P \Bigr) J\sigma_{+,m}(\eps)
  \end{equation}
  for $\sigma_{+,m}$ and
  \begin{equation}
    \label{transi2'}
    \frac{\delta_m}{\sqrt\eps} \bigl[ -a_m + b_m \e^{-\delta_m L} \bigr]
    = - J\sigma'_{-,m}(\eps)+\frac{1}{\eps}
                    \Bigl( \frac{n}{2}-P \Bigr) \sigma_{-,m}(\eps)
  \end{equation}
  for $\sigma_{-,m}$.
  
  In the remaining part of the proof, we want to show that
  $b_\infty=0$ using~\eqref{transi1'}.  The argument for $a_\infty=0$
  follows similarly from~\eqref{transi2'}.  Namely,
  from~\eqref{transi1'}, we obtain the additional information
  \begin{equation*}
     b_\infty
     = \lim_{m\to\infty} \frac{\sqrt\eps}{\delta_m} 
       \Bigl(
         - \gamma \eps^{-\gamma-1}                      d_{+,m}
         - \Bigl(\frac{n}{2}-P \Bigr) \eps^{-\gamma-1} Jd_{+,m} 
       \Bigr)
  \end{equation*}
  for $b_\infty$.  This equality combined with~\eqref{b1}
  give the necessary condition on $b_ \infty$, namely
  \begin{equation*}
    (\mu+\gamma)b_\infty
    = \begin{pmatrix} 
      \frac n 2-p+1 & 0\\
      0            & p-\frac n 2
    \end{pmatrix}
    b_\infty.
  \end{equation*}
  Using now the result of the following sublemma, we conclude that
  $b_\infty=0.$ Indeed we restrict ourself to the case where
  $\gamma>1/2$, so the last case $\gamma=1/2-\mu$ is not possible
  with $\mu\geq 0$.
\end{proof}
\begin{sublem}
  \label{souslem}
  Let $b$ be an eigenform of $A(A+1)$ with eigenvalue
  $\gamma(\gamma+1)$ relative to a non-zero eigenvalue $\mu^2$ of
  $\Delta_\Sigma$ and denote
  \begin{equation*}
     N_{\gamma,\mu,p}
     = \gamma + \mu + \Bigl( \frac n2 - P \Bigr) J
     = \gamma + \mu -
     \bigg( 
     \begin{array}{cc} \frac n2 - p + 1 & 0\\
       0                & p-\frac n2
     \end{array}
     \bigg).
   \end{equation*}
   The operator $N_{\gamma,\mu,p}$ restricted to $\mcH_j$ is
   identically $0$ iff $j=5$, $\gamma=\gamma_-(\mu^2)$ and
   $p=(n+1)/2$. In this case, $\mu\in {]}0,1]$ and $\gamma=1/2-\mu$.

   In all other cases, i.e., if $b\in\mcH_j$, $j=3,4$, or $b \in
   \mcH_5$ and $\gamma=\gamma_+(\mu^2)$ or $\gamma=\gamma_-(\mu^2)$
   but $p \ne (n+1)/2$ or $\mu>1$, then $N_{\gamma,\mu,p}(b)=0$
   implies $b=0$.
\end{sublem}
\begin{proof}
  We distinguish the three cases $b\in \mcH_3$,
  $b\in \mcH_4$ and $b\in \mcH_5$.

  If $b\in\mcH_3$ then we have
  \begin{equation*}
    b=\begin{pmatrix} b_1\\ 0 \end{pmatrix}
  \end{equation*}
  and
  \begin{equation*}
    \gamma = -\frac12 + \sqrt{\mu^2+\Bigl (\frac{n+3}{2}-p\Bigr)^2}.
  \end{equation*}
  So $N_{\gamma,\mu,p}(b)=0$ means 
  \begin{equation*}
    \bigg( \mu+ \sqrt{\mu^2+\Bigl( \frac{n+3}{2}-p \Bigr)^2} \bigg)
      b_{1} 
    = \Bigl( \frac{n+3}{2}-p \Bigr)b_{1}.
  \end{equation*}
  Since $\mu>0$, it follows that $b_{1}=0$.

  If $b\in\mcH_4$ then we have
  \begin{equation*}
    b= \begin{pmatrix} 0\\b_{2} \end{pmatrix}
  \end{equation*}
  and
  \begin{equation*}
    \gamma = -\frac12 + \sqrt{\mu^2+\Bigl(\frac{n-1}{2}-p \Bigr)^2}.
  \end{equation*}
  Consequently, we obtain if $N_{\gamma,\mu,p}(b)=0$,
  \begin{equation*}
    \bigg( \mu + \sqrt{\mu^2+\Bigl(\frac{n-1}{2}-p \Bigr)^2} \bigg)
      b_{2}
    = \Bigl( p-\frac{n-1}{2} \Bigr)b_{2},
  \end{equation*}
  hence $b_{2}=0$ since again, $\mu>0$.

  It remains to treat the case where $b \in \mcH_5$. Here, we
  have
  \begin{equation*}
    b= \begin{pmatrix} b_{1}\\b_{2} \end{pmatrix}.
  \end{equation*}
  Moreover, we know that $b_{1}=0$ if and only if
  $b_{2}=0$ due to the expression~\eqref{h5.ev} for the
  eigenforms.  In addition,
  \begin{equation*}
    \gamma_\pm 
    = -\frac12 
      + \bigg| \sqrt{\mu^2+\Bigl( \frac{n+1}{2}-p \Bigr)^2}\pm 1\bigg|.
  \end{equation*}
  Then, $N_{\gamma,\mu,p}(b)=0$ means
  \begin{equation*}
    \bigg( \mu + \bigg| \sqrt{\mu^2+\Bigl(\frac{n+1}{2}-p \Bigr)^2}
    \pm 1 \bigg| \bigg) b_{1}
    = \Bigl( \frac{n+3}{2}-p \Bigr) b_{1}
  \end{equation*}
   and also
   \begin{equation*}
     \bigg( \mu + \bigg| \sqrt{\mu^2+\Bigl(\frac{n+1}{2}-p\Bigr)^2}
     \pm 1 \bigg| \bigg) b_{2}
      = \Bigl( p-\frac{n-1}{2} \Bigr) b_{2}.
   \end{equation*}
   If $b\not=0$, we obtain first, that $\Bigl( p-\frac{n-1}{2} \Bigr)
   =\Bigl( \frac{n+3}{2}-p \Bigr)$, or $p=\frac{n+1}{2}$ (i.e.
   $a_{p}=0$) and secondly, that $\mu+|\mu\pm 1|=1$.

   So assume that $p=\frac{n+1}{2}.$ Since $\mu>0$, we have $\mu+|\mu
   +1|>1$ and as a consequence
   \begin{equation*}
     N_{\gamma_+(\mu^2),\mu,\frac{n+1}{2}}(b)=0
     \quad \Rightarrow \quad b=0.
   \end{equation*}

   As well, if $\mu>1,$ then $\mu+|\mu -1|=2\mu-1>1,$ so we have also
   \begin{equation*}
     \mu>1 \qquad \text{and} \qquad
     N_{\gamma_-(\mu^2),\mu,\frac{n+1}{2}}(b) = 0 
          \quad \Rightarrow \quad b=0.
   \end{equation*}

In the remaining case $\mu\in {]}0,1]$, we obtain $\mu+|\mu- 1|=1$
(and $\gamma_-(\mu^2)=1/2-\mu$), and
$N_{\gamma_-(\mu^2),\mu,(n+1)/2}=0$.
\end{proof}

 We study now the behaviour of the low energy forms $\phi_{m,\Lambda}$
 on the cones $\mcC_\eps^\pm$:
\begin{lem}
    \label{lem5}
  On $\mcC_\eps^\pm$, we have
  \begin{equation*}
    U_\pm\chi\phi_{m,\Lambda}=u_m+v_m
  \end{equation*}
  where
  \begin{equation*}
    \lim_{m\to\infty} \|v_m\|_{\Lsqr{\mcC_\eps^\pm}}=0
  \end{equation*}
  and $u_m$ is given as follows:
  \begin{enumerate}
  \item If $p \neq (n+1)/2$ or if there is no eigenvalue of $\Delta_{\Sigma}$
  for exact $p$-forms in the interval ${]}0,1{[}$, then
    \begin{equation*}
      u_m = \sum_{\gamma} 
         c_\gamma(m) t^{\gamma+1} F_\gamma(\lambda_m t^2)\sigma_\gamma
    \end{equation*}
    where the sum is finite over $\gamma\in [-\frac12,\infty{[}$, and the
    sequence $(c_\gamma(m))_m$ is bounded. Moreover, $F_\gamma \in \Ci{
      {[}0,\infty{[} } $ and
    $F_\gamma(0)=1$. Finally, $(\sigma_\gamma)_\gamma$ is
    independent of $t$ and is an orthonormal family in $\Lsqr
    {\Lambda^{p-1}T^*\Sigma} \oplus \Lsqr{\Lambda^{p}T^*\Sigma}$.
  \item If $p=(n+1)/2$ and if there is an eigenvalue $\mu^2$ of
    $\Delta_{\Sigma}$ for exact $p$-forms in the interval ${]}0,1{[}$,
    then we have
    \begin{equation*}
      u_m = \sum_{\gamma} 
         c_\gamma(m) t^{\gamma+1} F_\gamma(\lambda_m t^2)\sigma_\gamma
          + \sum_{\mu }
         c_\mu(m) t^{\mu-1/2} G_\gamma(\lambda_m t^2)\sigma_\mu,
    \end{equation*}
    where the first sum satisfies the same properties as in~(i) but we
    only have $\gamma\in \{-\frac12\}\cup [\frac12,\infty{[}$. In
    addition, the second sum is finite and runs over $\mu\in{]}0,1{[}$
    such that $\mu^2$ is an exact eigenvalue of $\Delta_{\Sigma}$,
    and the sequence $(c_\mu(m))_m$ is bounded.  In addition,
    $G_\gamma \in \Ci{ {[}0,\infty{[} } $ with $G_\gamma(0)=1$.
    Moreover, the family $\{U^{-1}_\pm \sigma_\gamma \}_\gamma \cup
    \{U^{-1}_\pm \sigma_\mu \}_\mu $ is orthonormal.
  \end{enumerate}
\end{lem}
\begin{proof}
  We continue with the same notation as in \lref{lem:low.ev}. We will
  only work on the behaviour on $\mcC_\eps^+$ since the other is
  similar. 
  We assume that $U_+(\phi_{m,\Lambda})=\sigma_m$ is a common
  eigenvector of both $A_ 0^2$ and $A^2+A$ for each $t$, i.e.,
  \begin{equation*}
    A_0^2\sigma_m=\mu^2\sigma_m 
         \qquad \text{and} \qquad 
    (A^2+A)\sigma_m=\gamma(\gamma+1)\sigma_m,
  \end{equation*}
  where we have dropped the subscript $+$.  The expression of
  $\sigma_m$ is given in~\eqref{bessel}.
  
  From~\eqref{growth}, \eqref{b1} and \lref{lem:low.ev} we conclude
  that for $\gamma>\frac12$ we have
  \begin{equation*}
    \|d_m g_\gamma\|^2\simeq h_\gamma(\eps) |d_m|^2=o(1),
  \end{equation*}
  if $m$ tends to $\infty$.

  We concentrate now on the case where $\gamma\in [-\frac12,\frac12]$.
  Equations~\eqref{transi1'} and~\eqref{transi1} imply, by elimination
  of $b_m$, that for a certain constant $c$ we have
   \begin{equation}
     \label{DN}
     \delta_m\sigma_m(\eps)-\sigma_m'(\eps) + \frac{1}{\eps} 
     \Bigl( \frac n2 - P \Bigr) J\sigma_m(\eps) =O\left(\e^{-c/\eps}\right).
   \end{equation}
   But from~\eqref{bessel} we conclude that for $\gamma\in
   {]}{-}\frac12,\frac12{[}$ we have
   \begin{multline*}
       \sigma_m'(\eps) 
       =  c_m \eps^\gamma (\gamma+1) F_\gamma(\lambda_m\eps^2)
        + 2 c_m \eps^{\gamma+2} \lambda_m F'_\gamma(\lambda_m \eps^2)\\
         {} -\gamma d_m \eps^{-\gamma-1} G_\gamma (\lambda_m\eps^2)
           + 2 d_m \eps^{-\gamma+1} \lambda_m G'_\gamma(\lambda_m\eps^2).
   \end{multline*}
    and if $\gamma=\pm1/2$ we obtain
    \begin{multline*}
        \sigma_m'(\eps) 
        = \eps^\gamma \bigl( 
             (c_m+ad_m\log\eps) (\gamma+1)+d_m a
          \bigr) F_\gamma(\lambda_m\eps^2) \\
         + 2 \eps^{\gamma+2}(c_m+d_m a \log\eps)  
                 \lambda_m F'_\gamma(\lambda_m \eps^2)\\
          {} -\gamma d_m \eps^{-\gamma-1} G_\gamma (\lambda_m\eps^2)
            + 2 d_m \eps^{-\gamma+1} \lambda_m G'_\gamma(\lambda_m\eps^2).
    \end{multline*}
   Note that $\delta_m$ also depends on $\eps$,
   namely, $\delta_m(\eps)=\mu/\eps +O(\eps)$. Therefore,
   equation~\eqref{DN} together with  the fact that the sequences
   $\{c_m\}_m$ and $\{d_m\}_m$ are bounded and the previous
   expressions for $\sigma_m'(\eps)$ leads to
   \begin{equation*}
     \Bigl( \gamma + \mu + \Bigl( \frac n2 - P \Bigr) J \Bigr)
     \eps^{-\gamma-1}G_\gamma(\lambda\eps^2) d_m
      =
      \begin{cases}
        O(\eps^\gamma)&\ \text{if $\gamma\ne \frac12$}\\
      O(\eps^{\frac12}|\log\eps|)&\ {\rm if\ } \gamma=\frac12
      \end{cases}
   \end{equation*}
   for $\gamma\in {]} {-}\frac12,\frac12{]}$ and
   \begin{equation*}
     \Bigl( \gamma + \mu + \Bigl( \frac n2 - P \Bigr) J \Bigr)
     \eps^{-\frac12}(\log\eps) aF_\gamma(\lambda\eps^2) d_m
      = O(\eps^{-\frac12})
   \end{equation*}
   for $\gamma=-1/2$.  Using the operator $N_{\gamma,\mu,p}$
   introduced in Sublemma~\ref{souslem} we have obtained
  \begin{equation*}
        N_{\gamma,\mu,p} (d_m)
       = \begin{cases}
           O(\eps^{1+2\gamma}) & \text{if $\gamma > -1/2$ 
              and $\gamma\not=1/2$}\\
	   O(\eps^{2}|\log\eps|) & \text{if $\gamma=1/2$}\\
           O(|\log\eps|^{-1})  & \text{if $\gamma = -1/2$}.
       \end{cases}
   \end{equation*}
   Hence, if the operator $N_{\gamma,\mu,p}$ is invertible, we have
   the same type of estimates for $d_m$ itself and
   $\lim_{m\to\infty}h_\gamma(\eps)|d_m|^2=0$ or
  \begin{equation*}
\lim_{m\to\infty}\|\sigma_m-c_m f_\gamma\|=0.
  \end{equation*}

  The result of Sublemma~\ref{souslem} shows that the operator
  $N_{\gamma,\mu,p}$ is invertible except in the case where $d_m$ is
  in $\mcH_5$, $p=(n+1)/2$ and $\gamma=\gamma_-(\mu^2)=1/2-\mu$.  The
  last equality imposes $\mu \in {]}0,1{]}$.  In particular, we have
  $\gamma\in {[}{-}\frac12,\frac12{[}$. Returning to
  Equation~\eqref{DN} we conclude then that
   \begin{equation*}
     \sigma_m'(\eps) + \frac{\frac12-\mu}{\eps} \sigma_m(\eps)
          + O(\eps) \sigma_m(\eps)
          = O(\e^{-c/\eps})
   \end{equation*}
   or equivalently
   \begin{equation*}
       \eps^{-\gamma} \frac{d}{dt} 
                 \Bigl( t^\gamma \sigma_m(t) \Bigr) \Bigl|_{t=\eps}\Bigr.
       + O(\eps)\sigma_m(\eps)
       = O(\e^{-c/\eps}).
   \end{equation*}
   Hence, if $\gamma=-1/2$, we obtain
   \begin{equation*}
     d_m = O(\eps^2\log \eps)
   \end{equation*}
   and as before $\lim_{m \to \infty} d_m=0$ and
   $\lim_{m\to\infty}\|\sigma_m-c_m f_\gamma\|=0$.\\
  If $\gamma \in {]}{-}\frac 12,\frac12{[}$, we have
   \begin{multline*}
       \eps^{-\gamma} \frac{d}{dt} 
            \Bigl( t^\gamma \sigma_m(t) \Bigr)\Bigl|_{t=\eps}\Bigr. 
       = c_m(2\gamma+1) \eps^\gamma
       F_\gamma(\lambda_m\eps^2)+
      c_m 2\eps^{\gamma+2} \lambda_mF'_\gamma(\lambda_m\eps^2)\\
       +\eps^{-\gamma+1} 2\lambda_mG'_\gamma(\lambda_m\eps^2)d_m.
   \end{multline*}
   and therefore
   \begin{equation*}
     c_m = O(\eps^{-2\gamma+1}), 
        \qquad \text{i.e.} \qquad
     \lim_{m\to\infty} c_m=0.
   \end{equation*}
 \end{proof}

\subsubsection{Harmonic terms}
It remains now to describe the behaviour of the harmonic components
$\phi_m^1$ and $\phi_m^2$.  We restrict our analysis to the space
$\mcH_1$, since $\mcH_1$ and $\mcH_2$ are dual by the Hodge-$*$
operator.

Again, we let $(\beta_{s,m},0)=U_s \chi\phi_m^1$ for $s=\emptyset,+,-$ be
the transformed pair of forms corresponding to the the handle and the
cones, respectively.  We know that on the handle, $\beta_m$ satisfies
the equation
\begin{equation}
  \label{EDO1}
  -\beta_m''=\lambda_m \beta_m \qquad \text{on $[0,L]$},
\end{equation}
whereas on the cones, $\beta_{\pm,m}$ fulfills
\begin{equation}
  \label{EDO2}
  -\beta_{\pm,m}''+\frac{\nu(\nu+1)}{t^2}
  \beta_{\pm,m}=\lambda_m\beta_{\pm,m},\hspace{1cm}\text{ with }\nu=n/2 - p +1.
\end{equation}

If $\nu\neq 0$ we put $\gamma=-\frac12+ \bigl| \frac{n+3}{2} -p
\bigr|=-\frac12+\bigl|\nu+ \frac 12 \bigr|$ as in \eqref{gam3} with $\mu=0$. 
The transmission conditions~\eqref{recolfq} and~\eqref{recolop} now reads as
\begin{equation}
  \label{trans1}
  \beta_m(L)=\beta_{+,m}(\eps),\qquad \beta_m(0)=-\beta_{-,m}(\eps)
\end{equation}
and
\begin{equation}
  \label{trans2}
  \beta_m'\Bigl(\frac{L\pm L}{2}\Bigr)
  = \beta_{\pm,m}'(\eps)+\frac  \nu \eps \beta_{\pm,m}(\eps)
  = \eps^{-\nu} \frac{d}{d t}
        \Bigl(t^\nu \beta_{\pm,\eps}(t)\Bigr)\Bigl|_{t=\eps}\Bigr..
\end{equation}

Since the $\Lsymb[2]$-norm of $\beta_{s,m}$ is bounded, it follows
from equation~\eqref{EDO1} and the transmission
conditions~\eqref{trans1} and~\eqref{trans2}, that $\beta_m(0)$,
$\beta_m(L)$, $\beta'_m(0)$, $\beta'_m(L)$ and $\beta_{\pm,m}(\eps)$
are all bounded sequences. Hence after passing to a subsequence, we
can assume that these sequences converge.  Moreover, from the
quadratic form expression~\eqref{locform}, we also know that there is
a uniform constant $C$ such that
\begin{equation}
  \label{fq}
  \sum_{s=\pm} \Bigl( \int_\eps^{1}
          \Bigl(|\beta'_{s,m}|^2 
          +\frac{\nu(\nu+1)}{t^2}  |\beta_{s,m}|^2\Bigr)dt
           -\frac 1 \eps \nu|\beta_{s,m}(\eps)|^2 \Bigr)
           + \int_0^L|\beta_m'|^2 dt \le C.
\end{equation}

We express the solutions of~\eqref{EDO2} as in \eqref{bessel}:
\begin{equation*}
 \beta_{\pm,m}(t) =c_{\pm,m} f_\gamma(t) + d_{\pm,m} g_\gamma(t).
\end{equation*}
As a consequence of the estimate \eqref{growth} on the 
$\Lsymb[2]$-norm of $\beta_{\pm,m},$  we obtain 
\begin{equation}
  \label{dinfini}
  c_{\pm,m} =O(1) 
     \qquad \text{and} \qquad 
  d_{\pm,m} =
  \begin{cases}
        O\bigl(\eps^{\gamma-1/2}  \bigr)  & \text{if $\gamma>1/2$,}\\
        O\bigl(|\log \eps |^{-1/2}\bigr)& \text{if $\gamma=1/2$,}\\
        O(1)                             & \text{if $\gamma<1/2$.}
  \end{cases}
\end{equation}
Again, after passing to a subsequence, we can assume that the
sequences $\{c_{\pm,m}\}_m$ converge to $c_{\pm,\infty}$.

But now from the transmission condition~\eqref{trans1} we know that
$\beta_{\pm,m}(\eps)$ and $c_{\pm,m}$ are bounded, and as a
consequence, $(d_m g_\gamma(\eps))_m$ is also bounded so
\begin{equation}
  \label{dinf} 
  \gamma \ge 0 \qquad \text{implies} \qquad 
   d_{\pm,m}=O(\eps^{\gamma}).
\end{equation}

In particular, we have 
\begin{cor} 
  \label{cor1}
  If $\gamma>0$, i.e., $\nu \notin\{-1,-1/2,0\}$, then
  \begin{equation*}
    \norm[{\Lsymb[2]}] {\beta_{\pm,m} - c_{\pm,\infty} f_\gamma} \to 0.
  \end{equation*}
\end{cor} 
\begin{proof}
  By the estimate~\eqref{growth}, there exists a constant $C>0$ such
  that
  \begin{equation*}
    \norm[{\Lsymb[2]}] {\beta_{\pm,m} -
      c_{\pm,\infty} f_\gamma}\leq C\sqrt{h_\gamma(\eps)}|d_{\pm,m}|.
  \end{equation*}
  By the preceding remark and~\eqref{growth}, we arrive at
  \begin{equation*}
    \norm[{\Lsymb[2]}] {\beta_{\pm,m} -  c_{\pm,\infty} f_\gamma}
    =
    \begin{cases}O(\sqrt{\eps})& \text{ if }\gamma>1/2\\
      O(\sqrt{\eps|\log\eps|}) & \text{ if }\gamma=1/2.
    \end{cases}
  \end{equation*}
\end{proof}
We study now the limit boundary conditions.
\begin{lem}
  \label{lem7}
  If $p-1\geq \frac{n+1}{2}$ or $\nu \le -1/2$ (and therefore
  $\gamma=-1-\nu$), then we obtain, at the limit, the Dirichlet
  boundary conditions:
  \begin{equation*}
    \lim_{m\to\infty} \beta_m(0)=0
     \qquad \text{and} \qquad
    \lim_{m\to\infty} \beta_m(L)=0.
  \end{equation*}
\end{lem}

\begin{proof}
  If $\nu<-1/2$ then $\nu(\nu+1)\geq 0$ and the estimate of the
  quadratic form~\eqref{fq} gives
  \begin{equation*}
    \beta_\eps(0) = -\beta_{-,\eps}(\eps)=O(\sqrt\eps), \qquad
    \beta_\eps(L) =  \beta_{+,\eps}(\eps)=O(\sqrt\eps).
  \end{equation*}
 Now suppose that $\nu=-1/2$. In this case the
  estimate~\eqref{fq} gives
  \begin{equation*}
    \int_\eps^{1} |\beta_{\pm,m}'(t)|^2dt
     - \int_\eps^{1} \frac{1} {4t^2}| \beta_{\pm,m}(t)|^2dt
     + \frac{1}{2\eps} |\beta_{\pm,m}(\eps)|^2 
    \le C.
  \end{equation*}
  However, after integration by parts, the left hand side of this
  inequality is
  \begin{equation*}
    \int_\eps^1 t \Bigl| \frac{d}{dt} 
                 \Bigl(t^{-1/2}\beta_{\pm,m}(t)\Bigr) \Bigr|^2 dt.
  \end{equation*}
  Let $v \in \Ccic{{[}\eps,1{[}}$ and let $\varphi(t)=\sqrt{|\log t|}\,
  v(t)$, then we have
  \begin{align*}
    \int_\eps^1 t |\varphi'(t)|^2 dt
    & = \int_\eps^1|v(t)|^2  \frac{dt}{4t|\log t|}
      + \int_\eps^1|v'(t)|^2 |\log(t)|dt
      - \int_\eps^1v(t) v'(t) \, dt\\
    &\ge -\int_\eps^1 v(t) v'(t) \, dt\\
    &= \frac12 |v(\eps)|^2.
  \end{align*}
  Applying the former estimate with $v(t)=(t |\log t|)^{-1/2}
  \beta(t)$, we obtain
  \begin{equation*}
    \beta_{\pm,m}(\eps)=O\bigl(\sqrt{\eps|\log \eps|}\bigr)
  \end{equation*}
  which proves the claim.
\end{proof}

We focus now on the case where $\nu>0$ and consequently $\gamma=\nu$.
\begin{lem}
 \label{lem8}
 If $p-1\leq\frac{n-1}{2},$ or $\nu\ge \frac12$ (and therefore
 $\gamma=\nu$), then we obtain, at the limit, the Neumann boundary
 conditions
  \begin{equation*}
    \lim_{m\to\infty} \beta'_m(0)=0
         \qquad \text{and} \qquad
    \lim_{m\to\infty} \beta'_m(L)=0.
  \end{equation*}
\end{lem}
\begin{proof}
  The first order transmission conditions~\eqref{trans2} imply that we
  have to look at the limit of the sequence formed by
  \begin{equation*}
    \eps^{-\nu} \frac{d}{d t}
                  \Bigl(t^\nu \beta_{\pm,\eps}(t)\Bigr)\Bigl|_{t=\eps}\Bigr..
  \end{equation*}
  But the limit of this sequence is
  \begin{equation*}
    \begin{cases}
      \lim_{m\to\infty} d_{\pm,m}\eps^{1-\nu}\lambda_m
                G'_{\gamma}(\lambda_m\eps^2)
            & \text{if $\nu\geq 1$,}\\
      \lim_{m\to\infty} 2 a d_{\pm,m}\eps^{1/2}(\log \eps)
                F_{\gamma}(\lambda_m\eps^2)
            & \text{if $\nu=1/2$.}
    \end{cases}
  \end{equation*}
    Now following~\eqref{dinf}, $d_{\pm,m}=O(\eps^\nu)$, and we obtain
  finally
  \begin{equation*}
    \Bigl|\beta_m'\Bigl(\frac{L\pm L}{2}\Bigr)\Bigr|
    = \begin{cases}
        O(\eps)             & \text{if $\nu>1/2$,}\\
        O(\eps|\log \eps|) & \text{if $\nu=1/2$}
    \end{cases}
  \end{equation*}
  and the result follows.
\end{proof}

\begin{cor}
  \label{cor:hp} 
  If $H^{n/2}(\Sigma)=0$, then we have
  \begin{equation*}
    U(\phi_{m,1})=(c_{\pm,m} f_\gamma,0)+r_m
  \end{equation*}
  on $\mcC_\eps^\pm$, where
  \begin{equation*}
    \lim_{m\to\infty} \|r_m\|_{\Lsymb[2]}=0,
  \end{equation*}
  the sequence $c_{\pm,m}$ converges to $c_{\pm,\infty}$ and
  $f_\gamma$ is given by \eqref{bessel} with $\gamma=-\frac12+ \bigl|
  \frac{n+3}{2} -p\bigr|$.
\end{cor}
\begin{proof} 
  With the preceding notations we have to show that
  \begin{equation*}
    \norm[{\Lsymb[2]}] {\beta_{\pm,m} - c_{\pm,\infty} f_\gamma} \to 0.
  \end{equation*}
  Recall that $\gamma=-\frac12+\bigl|\nu+ \frac 12 \bigr|$ with
  $\nu=n/2 - p +1$. \cref{cor1} fulfills the case $\gamma > 0$. If
  $\gamma\leq 0$ then, by the estimate~\eqref{growth},
  $\norm[{\Lsymb[2]}] {\beta_{\pm,m} - c_{\pm,\infty} f_\gamma}$ is
  controlled by $|d_{\pm,m}|$. It remains to show that
  $\lim_{m\to\infty}d_{\pm,m}=0$ if $\gamma\leq 0$ (and $\nu\neq 0$ by
  hypothesis).

  The case $\gamma=0$ corresponds only to $\nu=-1$.  The boundedness
  of the quadratic form gives then $\beta_{\pm,m}(\eps)=O(\sqrt\eps)$.
  But this implies, by the expression of the solutions
  of~\eqref{EDO2}, that
  \begin{equation*}
    d_{\pm,m}=O(\sqrt\eps).
  \end{equation*}
    
  The case $\gamma=-1/2$ corresponds to $\nu=-1/2$. In this case we
  have already seen that
  \begin{equation*}
    \beta_{\pm,m} (\eps) = O\bigl(\sqrt{\eps|\log \eps|}\bigr).
  \end{equation*}
  The expression of the solutions of~\eqref{EDO2} for $\gamma=-1/2$
  implies that
  \begin{equation*}
    d_{\pm,m} = O \bigg( \frac{1}{\sqrt{|\log \eps|}} \bigg)
  \end{equation*}
  and the result follows.
\end{proof}

\begin{rem}
  \label{nu=0}
  If the cohomology group $H^{n/2}(\Sigma)$ is non-trivial, what
  happens in the case $\nu=0$, i.e., for forms of degree $p=n/2+1$?
  The quadratic form~\eqref{fq} becomes
  \begin{equation*}
    \sum_{s=\pm} \int_{\eps}^1|\beta'_{s,m}|^2+\int_0^L|\beta_m'|^2.
  \end{equation*}
  Actually, we are just on intervals and the transmission condition
  gives the limit situation.  From the sequence $\{\beta_m\}_m$, which
  is bounded in $\Sobsymb[1]$ on the global interval, one can extract
  a sequence which converges to an eigenform on $\overline M$ with
  eigenvalue $\lambda$, and the boundary values $\beta_\pm(0)$ and
  $\beta'_\pm(0)$ must satisfy the transmission conditions
  \begin{equation}
    \label{tran}
    \beta_- (0) =-\beta(0), \qquad
    \beta_+ (0) = \beta(L); \qquad
    \beta'_- (0)=\beta'(0), \qquad
    \beta'_+ (0)=\beta'(L)
  \end{equation}
  for $\beta$ satisfying $-\beta''=\lambda\beta$ on $[0,L]$.

  For instance if we come from the situation where $M=\R/\Z\times \Sigma$ is a
  $3$-torus and $\Sigma=\R^2/\Z^2 $ a generating torus, the limit problem 
  described here is \emph{not} decoupling.
\end{rem}

\section{The limit problem}
\label{sec:limit}
 We first recall the results of~\cite{BS} and~\cite{L} concerning the
 closed extensions of the operator $D=d+d^\ast$ on the manifold with
 conical singularities $\overline M$.  They are classified by the
 spectrum of its \emph{Mellin symbol}, which is here the operator with
 parameter $A+z$. In our case, we need two copies of $A+z$, since we have 
 two conical singularities.  Recall that $A$ is
 the operator defined in~\eqref{opeA} by
 \begin{equation*}
   A = \begin{pmatrix} 
           \dfrac n 2 -P &-D_0\\
           -D_0          & P - \dfrac{n}{2} 
       \end{pmatrix}.
 \end{equation*}
 If $\spec (A) \cap\ {]}{-}\frac12,\frac12{[}$ is empty then $D_{\max}
 = D_{\min}$. In particular, $D$ is essentially selfadjoint on the
 space of smooth functions with compact support \emph{away} from the
 conical singularities.
 Otherwise, the quotient $\dom(D_{\max})/\dom(D_{\min})$ is isomorphic
 to
 \begin{equation*}
   B_+ \oplus B_- \qquad \text{where} \qquad
   B_\pm := \bigoplus_{\gamma \in {]}{-}\frac12,\frac12{[}} 
              \Ker(A-\gamma).
 \end{equation*}
 More precisely, by Lemma~3.2 of~\cite{BS}, there is a surjective
 linear map
 \begin{equation*}
   \mcL = \mcL_+ \oplus \mcL_- \colon 
          \dom(D_{\max}) \rightarrow B_+\oplus B_-
 \end{equation*}
 with $\ker \mcL =\dom(D_{\min})$. Furthermore, we have the estimate
 \begin{equation*}
   \| u_\pm(t) - t^{-A} \mcL_\pm(\phi) \|_{\Lsqr{\Sigma}} 
     \leq C(\phi) |t\log t|^{1/2}
 \end{equation*}
 for $\phi\in\dom(D_{\max})$ and $u_\pm=U_\pm(\phi)$, where $U_\pm$ is
 defined in \sref{sec:lap}.

 Now to any subspace $W\subset B_+\oplus B_-$, we associate the
 operator $D_W$ with domain $\dom(D_W):=\mcL^{-1}(W)$. As a result
 of~\cite{BS}, all closed extensions of $D_{\min}$ are obtained by this
 way.  Remark that each $D_W$ defines a selfadjoint extension
 $(D_W)^\ast \circ D_W$ of the Hodge-Laplace operator, and we have
 $(D_W)^\ast=D_{\I (W^\perp)}$, where
 \begin{equation*}
   \I = \begin{pmatrix} 
                0 & \id\\
               -\id & 0   
        \end{pmatrix}, 
       \qquad \text{ie., \quad$\I(\beta, \alpha)=(\alpha,-\beta)$}.
 \end{equation*}
 \sloppy This extension is associated to the quadratic form
 $\phi\mapsto \|D\phi\|^2$ with domain $\dom(D_W)$.  We have already
 computed the spectrum of the operator $A^2+A$ restricted to the
 spaces $\mcH_1$, \dots, $\mcH_5$ in \sref{specA}.  It is expressed
 for each space $\mcH_i$ in the form $\gamma(\gamma+1)$ with $\gamma
 \ge -1/2$, where $\gamma$ is given in~\eqref{gam5}
 and~\eqref{gam3}--\eqref{gam2}.

 Hence the spectrum of $A$ is among the values $\gamma_\pm,
 -1-\gamma_\pm$ where $\gamma_\pm$ is given by~\eqref{gam5}, and the
 $\gamma, -1-\gamma$, for the $\gamma$ appearing
 in~\eqref{gam3}--\eqref{gam2}.\footnote{Using~\cite{BS}, we can calculate 
   explicitely the
   spectrum of $A$. In fact, $\spec (A)$ consists of the values
   $\gamma=\pm\frac 12\pm\sqrt{\mu^2+a^2_{p+1}}$, where $\mu^2$ runs
   over the spectrum of $\Delta_\Sigma$ acting on co-closed
   $p$-forms.}  We have to take care of the fact that the spaces are not 
 all stable under the action  
 of $A.$ Indeed $\mcH_1$ and $\mcH_2$ are stable by the action of
 $A$ and consequently the spectrum of $A$ contains $\frac n2-p+1 $ with
 multiplicity $b_{p-1}(\Sigma)$ and $p-\frac n2$ with multiplicity
 $b_{p}(\Sigma)$ where $p$ runs over $0,...,n$, $\mcH_5$ also is
 stable by $A$, but $\mcH_3$ and $\mcH_4$ are not.
 Nevertheless we remark that $A$ satisfies the relation
 $A\circ\I=-\I\circ A,$ and, if $Au=\gamma u$ with $u\in\mcH_5,$ then
 $A(\I u)=-\gamma\I u$ with $\I u$ living in the $\mcH_3 \oplus \mcH_4$
 components of {\em other} degrees.  Then, considering all the degrees
 together, $\sum_p(\mcH_3 \oplus \mcH_4)$ is stable under the action of $A$
 and its spectrum on this component is the opposite of its spectrum on $\mcH_5$.
 Thus the spectrum of $A$ is determined by its
 restriction on $\mcH_1$, $\mcH_2$, and $\mcH_5$.
  
For our concern we have the following result:
 \begin{lem}
   \label{lem6}
   Let $n$ be odd. Then the eigenvalues $\gamma$ of $A$ restricted to
   $\mcH_5$ with $\gamma \in {]}{-}\frac12,\frac12{[}$ are precisely
   the values $\gamma=\gamma_-(\mu^2)$ entering in the description of
   the spectrum of $A(A+1)$ with $a_p=0$, i.e., $p=(n+1)/2$, and
   $\mu\in{]}0,1{[},$ thus $\gamma_-(\mu^2)=\mu-\frac12.$
 \end{lem}
 \begin{proof}
   Let $\sigma=(\beta,\alpha)\in\mcH_5$  and denote the degree of
   $\alpha$ by $p$. Then $A$ is given by
   \begin{equation*}
     A = \begin{pmatrix} p-1-\frac{n}{2}&-d^\ast_0\\
                 -d_0&\frac{n}{2}-p\end{pmatrix},
   \end{equation*}
   so that $A\sigma =\gamma\sigma$ is equivalent to
   \begin{gather*}
     \Bigl(p-1-\frac{n}{2}-\gamma \Bigr) \beta
     =d_0^\ast\alpha 
            \qquad \text{and} \qquad
     \Bigl(\frac{n}{2}-p-\gamma \Bigr)\alpha=d_0\beta, \quad \text{i.e.,}\\
     \Delta_\Sigma\alpha
     =\Big( \bigl( \frac 12 + \gamma \bigr)^2 - a_p^2\Big) \alpha
            \qquad \text{and} \qquad
     \Delta_\Sigma\beta
     =\Big( \bigl( \frac 12+\gamma \bigr)^2 - a_p^2 \Big) \beta.
   \end{gather*}
   In particular, the latter equalities mean that there exists
   \begin{equation*}
     \mu^2 \in \spec (\Delta^p_{\Sigma,\mathrm c}) \cap 
               \spec (\Delta^{p-1}_{\Sigma,\mathrm{cc}})
     \qquad \text{with} \qquad 
     \gamma=-\frac 12\pm\sqrt{\mu^2+a_p^2}
   \end{equation*}
   where $\Delta_{\Sigma,\mathrm c}$ resp.\
   $\Delta_{\Sigma,\mathrm{cc}}$ denotes the Laplacian acting on closed
   resp.\ co-closed, forms.
   Now, $\gamma \in {]}{-} \frac 12, \frac 12{[}$ implies
   $\gamma=-\frac 12+\sqrt{\mu^2+a_p^2}$ with $a_p=0$ and $\mu^2\in
   {]}0,1{[}$ or with $a_p=\pm \frac 12$ and $\mu^2\in [0,\frac 34{[}$.
   Since we assumed that $n$ is odd, the unique possibility is $a_p=0$,
   i.e., $p=(n+1)/2$, and therefore $\gamma=\mu-\frac 12$ since
   $\mu^2\in {]}0,1{[}$.
   Reciproquely, if $a_p=0$ and $\Delta^p_{\Sigma,\mathrm c}
   \alpha=\mu^2\alpha$ with $\mu^2\in {]}0,1{[}$, then 
 $A(\beta,\alpha)=(\mu-\frac12)(\beta,\alpha)$ with
   $\beta=-d_0^\ast \alpha/\mu$, and also $\Delta^{p-1}_{\Sigma,
     \mathrm{cc}} \beta =\mu^2\beta.$ Then $\sigma=\sigma_\mu,$ with
   the notations of~\lref{lem5}.
 \end{proof}
 In fact we have proved more, namely: $\spec(A) \cap
 {]}{-}\frac12,\frac12{[}=\emptyset$ if and only if
 \begin{itemize} 
 \item the spectrum of $\Delta_\Sigma$ on exact $(n+1)/2$-forms (or
   co-exact $(n-1)/2$-forms) does not meet the interval $]0,1[$, for
   $n$ odd,
 \item the spectrum of $\Delta_\Sigma$ on $n/2$-forms does not meet the
   interval $[0,\frac 34{[}$, for $n$ even.
 \end{itemize}
 Indeed, in the last alternative, for $a_p=\pm\frac12$, i.e.,
 $p\in\{\frac n2, \frac n2 +1\}$, if $\Delta^p_{\Sigma,\mathrm c} \alpha
 =\mu^2\alpha$ with $\mu\neq 0$, then $\gamma=-\frac
 12+\sqrt{\mu^2+\frac 14}>0$ and again
 $A(\beta,\alpha)=\gamma(\beta,\alpha)$ with $\beta=-\frac{d_0^\ast
   \alpha}{\frac 12\pm\frac 12+\gamma}$, and also
 $\Delta^{p-1}_{\Sigma,\mathrm{cc}} \beta =\mu^2\beta$; we remark that
 either $\alpha$ is an exact $n/2$-form, or $\beta$ is a co-exact
 $n/2$-form.  The case $\mu=0$ corresponds in fact to components in
 $\mcH_1$ for $p=\frac n2 +1$ and $\mcH_2$ for $p=\frac n2.$ They have
 already been described and correspond to $n/2$-harmonic forms.

We can now describe the extensions of the Laplacian obtained for the limit 
operator; they depend on $p$:
\begin{itemize}
\item If $p\not\in \{(n+1)/2, n/2,n/2+1\}$ or if $p\in \{ n/2,n/2+1\}$
  and $b_{n/2}(\Sigma)=0$ or if $p=(n+1)/2$ and the spectrum of
  $\Delta_\Sigma$ on exact $(n+1)/2$-forms does not meet the interval
  $]0,1[$, then on the manifold part $\overline M$ the limit operator
  is the Friedrichs extension of the Hodge-Laplace operator, that is
  $D_{\max}\circ D_{\min}$ restricted to $\Lsqr{\Lambda^pT^*M}$. It is
  the Friedrichs extension of the Laplacian defined by the quadratic
  form $\sigma\mapsto \|D\sigma\|^2$ with domain $\dom(D_{\min})$.
\item If $p=(n+1)/2$ and $\Delta_\Sigma$ has exact
  $(n+1)/2$-eigenvalues in the interval $]0,1[$, then the limit
  operator is (on the manifold part) $D_{\min}\circ D_{\max}$
  restricted to $\Lsqr{\Lambda^{\frac{n+1}{2}}T^* M}$.  It is the
  Friedrichs extension of the Laplacian defined by the quadratic form
  $\phi\mapsto \|D\phi\|^2$ with domain $\dom(D_{\max})$.
\item In the case when $H^{n/2}(\Sigma)\neq \{0\},$ and $p=\frac n2 $
  or $p=\frac n2+1$, the limit operator does not come from a
  selfadjoint extension of the Hodge-Laplace operator for the conical
  manifold $\overline M$ but from a selfadjoint extension of an
  operator acting on
  \begin{equation*}
    \Ccic{\Lambda^p T^* \overline M\setminus S,g_0 } \oplus
    \Ccic{{]}0,L{[},\mcH^{p-1}(\Sigma) \oplus \mcH^p(\Sigma)}\ ,
  \end{equation*}
  where $\mcH^p(\Sigma)$ denotes the space of harmonic $p$-forms on
  $\Sigma$ and $S$ is the singular part of $\overline M$, that is two
  points corresponding to the shrunken manifold $\Sigma$ at the tip of
  each cone. This operator acts as the Laplacian on the first
  component and by $-d^2/dt^2$ on the last component.
  \begin{itemize}
  \item Suppose that $p=n/2$, then the limit operator is associated to
    the quadratic form
    \begin{equation*}
      (\varphi,\sigma)\mapsto q(\varphi,\sigma):=\int_{\overline{M}} 
      \left(|d\varphi|^2+|d^*\varphi|^2\right)d\vol+\int_0^L
      |\sigma'(t)|dt
    \end{equation*}
    with the domain $\dom(q)$ where $(\varphi,\sigma)\in \dom(q)$ if
    and only if the following conditions are satisfied:
    \begin{equation}
    \label{eq:defn2}
      \begin{split}
         & \varphi \in 
            \Lsqr{\Lambda^{n/2} T^*\overline M}\cap \dom(D_{\max})  \\ 
         & \mcL_\pm(\varphi)=(0,\alpha_\pm)\in 
                  \{0\}\oplus \mcH^{n/2}(\Sigma )\subset\ker A \\ 
         & \sigma=(\beta,\alpha)\in 
               \bigSob{[0,L],\mcH^{n/2-1}(\Sigma) \oplus 
                      \mcH^{n/2}(\Sigma)} \\ 
         & \alpha_-=\alpha(0) \mbox{ and } \alpha_+=\alpha(L)
       \end{split}
    \end{equation}
   \item Suppose that $p=n/2+1$, then the limit operator is associated
     to the quadratic form
    \begin{equation*}
      (\varphi,\sigma)\mapsto q(\varphi,\sigma):=\int_{\overline{M}} 
        \left(|d\varphi|^2+|d^*\varphi|^2\right)d\vol+\int_0^L
            |\sigma'(t)|dt
    \end{equation*}
    with the domain $\dom(q)$ where $(\varphi, \sigma)\in \dom(q)$ if
    and only if the following conditions are satisfied:
    \begin{equation}
      \label{eq:defn2pu}
      \begin{split}
        &\varphi\in \Lsqr{\Lambda^{n/2+1} 
                T^*\overline M}\cap \dom(D_{\max})\\
        &\mcL_\pm(\varphi)=(\beta_\pm,0) \in 
                \mcH^{n/2}(\Sigma)\oplus \{0\}\subset\ker A\\
        &\sigma=(\beta,\alpha)\in \bigSob{[0,L],\mcH^{n/2-1}(\Sigma)
                        \oplus \mcH^{n/2}(\Sigma)}\\
        & \beta_-=-\beta(0) \mbox{ and }\beta_+=\beta(L).
      \end{split} 
    \end{equation}
  \end{itemize}
\end{itemize} 
\subsection*{Proof of \tref{thmC}}

We are now able to prove our main convergence result, namely
\tref{thmC}. More generally, we show the following:
\begin{thm}
  \label{sansgap}
  If we drop the condition $H^{n/2}(\Sigma)= 0$, the convergence
  results are the same as in \tref{thmC} except for the degrees
  $p=n/2$ and $p=n/2+1$ where the spectrum of the Hodge-de Rham
  operator of the manifold $M_\eps$ acting on these $p$-forms
  converges to the spectrum of the limit problem described
  in~\eqref{eq:defn2}--\eqref{eq:defn2pu}.
\end{thm}
\begin{proof}
  By duality it is sufficient to consider $p<n/2+1$. Let $\{\mu_N\}$,
  $N\geq 1$, be the sequence of the eigenvalues, counted with
  multiplicity, of the limit operator as described by the theorem in
  this degree.

  \subsection*{Upper bound} We show first that $\limsup_{\eps\to
    0}\lambda^p_N(\eps)\leq\mu_N$ by transplanting the corresponding
  eigenforms on $M_\eps$. The formula is then just a consequence of
  the minimax formula. Let us describe how the different type of
  eigenforms are transplanted.
  
  \subsubsection*{Eigenforms in $\dom(D_{\min})$ on $\overline M$.}
  These are the easiest because if $\phi\in \dom(D_{\min})$ then by
  definition, we find a sequence $\phi_l\in \Ccic{\Lambda^\bullet
    T^*\overline M\setminus S}$ such that
    \begin{equation*}
      \lim_{l\to\infty} \|\phi_l-\phi\|+\|D\phi_l-D\phi\|=0.
    \end{equation*}
    These $\phi_l$ are transplanted easily on the manifold
    $M_{\eps_l}$.
    
    \subsubsection* {Eigenforms of $ D_{\max}\cap
      \Lsqr{\Lambda^{(n+1)/2} T^*\overline M} $ on $\overline M$}
  Any such form $\phi$ can be written as
  \begin{equation*}
    \phi=\phi_0+\phi_1=\phi_0+\phi_++\phi_-
  \end{equation*}
  where $\phi_0\in \dom(D_{\min})$ and $\phi_+,\phi_-$ have support in
  $\mcC^\pm_0$ and
  \begin{equation*}
    U_\pm(\phi_\pm)=\sum_{\gamma} t^{-\gamma}c_{\pm,\gamma} \sigma_\gamma
  \end{equation*}
  where $c_{\pm,\gamma}\in \C$ and each $\sigma_{\gamma}\in \mcH_5$
  satisfies $A \sigma_{\gamma} =\gamma \sigma_{\gamma}$ for a
  $\gamma\in {]}{-}1/2,1/2{[}$ associated to $\mu_\gamma$ an exact
  $p$-eigenvalue of $\Delta_\Sigma$.  We only need to explain how we
  construct the transplantation $\phi_{1,\eps}$ of $\phi_1$ on
  $M_\eps.$
 
  On $M_\eps\setminus \mcA_\eps$ we let $\phi_{1,\eps}=\phi_{1}$ and
  on the handle $\mcA_\eps$, we define $\phi_{1,\eps}$ by
  \begin{equation*}
   U( \phi_{1,\eps}) = \sum_{\gamma} \eps^{-\gamma}
          \Big(c_{+,\gamma}\sigma_\gamma \chi_0(L-t)
               \e^{-\frac{\mu_\gamma}{\eps}(L-t)} 
              +Jc_{-,\gamma}\sigma_\gamma\chi_0(t)
               \e^{-\frac{\mu_\gamma}{\eps}(t)}\Big),
  \end{equation*}
  where $\chi_0$ is a cut-off function $\chi_0$ which satisfies
  $\chi_0(t)=1$ for $0<t<L/4$ and $\chi_0(t)=0$ for $L/3<t$.

  It is an easy calculation to show that $\int_{\mcA_\eps}
  \bigl(|D\phi_{1,\eps}|^2+|\phi_{1,\eps}|^2)\bigr)=O(\eps^{1-2\delta})$, 
  for a certain $\delta \in  {]}{-}1/2,1/2[$. 
  
\subsubsection*{Eigenforms of the interval with harmonic values in
    $\Sigma$}
  If we express the forms in terms of $\sigma$, as described at the
  beginning of \sref{sec:lap}, the Dirichlet spectrum of the interval
  corresponds to a form like $(0,f(t)\alpha)$ with $\alpha$ a $p$-form
  harmonic on $\Sigma$ and $f$ an eigenfunction for the Dirichlet
  Laplacian on the interval, it can be prolongated by 0.  The Neumann
  spectrum of the interval corresponds to a form like $(f(t)\beta,0)$
  with $\beta$ a $(p-1)$-form harmonic on $\Sigma$ and $f$ an
  eigenfunction for the Neumann Laplacian on the interval (or its dual
  by the Hodge-$\ast$ operator in the case $p=(n+1)/2$) it can be
  prolongated by 
  \begin{equation*}
    \sigma_+=(\eps^{n/2+1-p}\xi(t)f(L)t^{p-1-n/2}\beta,0)
  \end{equation*}
  where $\xi$ is a fixed cut-off function, $0\le \xi\le 1$, $\xi=1$ on
  $[0,1/4]$ and $\xi=0$ on $[3/4,1],$ and with the same type of
  expression on the other end.  The $q$-norm of the prolongation given
  here is of order $O(\sqrt\eps)$, except in the case $p=(n+1)/2$
  where we obtain $O(\sqrt{\eps|\log\eps|})$ (the calculus is the same
  as in~\cite[Eq.~(2.1)]{AC2}).

  \subsubsection*{Special case $H^{n/2}(\Sigma)\neq 0$}
  In this case, the eigenforms of degree $p=n/2$ belonging to the
  limit problem, can be transplanted as follows.  Let
  $(\varphi,\sigma)\in \dom(q)$; we know that, as before,
  $\varphi=\varphi_0+\varphi_1=\varphi_0+\varphi_+ +\varphi_-$ where
  $\varphi_0\in \dom(D_{\min})$ and $\varphi_+,\varphi_-$ have support
  in $\mcC^\pm_0$ and $U_\pm(\varphi_\pm)=(0, \alpha_\pm)$ where
  $\alpha \in \mcH^{n/2}(\Sigma)$ is constant on $[0,1/2]$ and
  $\sigma=(\beta,\alpha)\in H^1([0,L],
  \mcH^{n/2-1}(\Sigma)\oplus\mcH^{n/2}(\Sigma))$ satisfies
  $\alpha(0)=\alpha_-$ and $\alpha(L)=\alpha_+$.

  We extend $\beta$ as before for the Neumann spectrum of the
  interval, and because $U_\pm(\varphi_\pm)$ is constant on $[0,1/2]$
  it is easy to transplant $(\varphi_1,\sigma)$ on $M_\eps$ for
  $\eps\le1/2$.

  \subsubsection*{Conclusion}
  Now, for any rank $N$, let $\phi_1,\dots,\phi_N$ be an orthonormal
  basis of the total eigenspace $E_N$ of the limit problem,
  corresponding to the $N$ first eigenvalues. For any $\eps> 0$ we
  define a linear operator $T_\eps$ from $E_N$ to the domain of the
  quadratic form on $M_\eps$ by demanding that $T_\eps(\phi_j)$ is the
  transplanted form as described above.  The preceding estimates show
  that $\langle
  T_\eps(\phi),T_\eps(\psi)\rangle=\langle\phi,\psi\rangle+o(1)$ and
  also that $q(T_\eps(\phi),T_\eps(\psi))=q(\phi,\psi)+o(1)$.
  Evaluating the Rayleigh-Ritz quotient on the image $T_\eps(E_N)$
  gives then, with the minimax formula,
  \begin{equation*}
    \lambda^p_N(\eps)\leq\mu_N+o(1).
  \end{equation*}

  \subsection*{Lower bound}
  To show the other inequality, namely $\liminf_{\eps\to
    0}\lambda^p_N(\eps)\geq\mu_N$, we use the estimates provided in
  \sref{sec:asymp}. The eigenvalues inequality is then a consequence
  of the minimax principle applied to the limit of a subsequence for
  an orthonormal family of the $N$ first eigenforms of $M_\eps$.

  We give the argument first for one eigenvalue. For simplicity, we
  assume that $H^{n/2}(\Sigma)=0$.  The same
  proof, with a slight modification of the arguments, also works in
  the case when $H^{n/2}(\Sigma)$ is non-trivial.

  We consider a subsequence $\lambda_m=\lambda^p_N(\eps_m)$ such that
  \begin{equation*}
    \lim_{m \to \infty} \lambda_m 
    = \liminf_{\eps\to 0}\lambda^p_N(\eps)=\lambda
  \end{equation*}
  and denote the corresponding normalised eigenforms by $\phi_m$,
  namely, (in the following we write $\eps=\eps_m$)
  \begin{equation*}
    D^2 \phi_m = \lambda_m \phi_m 
    \qquad \text{and} \qquad 
    \|\phi_m\|_{\Lsymb[2]}=1.
  \end{equation*}
  For $\eps>0$ small enough, we will construct a form
  \begin{equation*}
    \psi_m \in \Lsqr{\Lambda^p T^*\overline M} \oplus 
               \Lsqr{[0,L], \mcH^{p-1}(\Sigma) \oplus \mcH^p(\Sigma)}
  \end{equation*}
  which is in the domain of the quadratic form of the associated limit
  operator. Here again, $\mcH^p(\Sigma)$ denotes the space of harmonic
  $p$-forms on $(\Sigma,h)$. Recall that we have denoted the spectrum
  of the limit operator by $\{\mu_N\}_N$. Moreover, the correspondence
  $\phi_m \mapsto \psi_m$ will be an almost isometry.  We begin to
  define $\psi_m$ (or more precisely $U\psi_m$ on $[0,L]$).  From
  \lrefs{lem:high.ev}{lem:low.ev} we conclude that on the handle
  \begin{equation*}
    \phi_m=h_m+k_m
  \end{equation*}  
  where $\lim_{m\to\infty}\|k_m\|_{\Lsymb[2]}=0$, $D^2
  h_m=\lambda_m h_m$  and $h_m$ is transversally harmonic.

  Moreover, by \lrefs{lem7}{lem8}, we can decompose
  $h_m=h^\Dir_{m}+h^\Neu_{m}$ where
  \begin{equation*}
    Uh^\Dir_{m} \Bigl(\frac{L\pm L}{2}\Bigr)
    =O\bigl(\sqrt{\eps|\log\eps|}\bigr)
  \end{equation*}
  and
  \begin{equation*}
    \frac{d}{dt} (U h^\Neu_{m}) \Bigl(\frac{L\pm L}{2}\Bigr)
    =O(\eps|\log\eps|).
  \end{equation*}
  Since $h_m$ satisfies the eigenvalue equation, we conclude that
  $u^\Dir_{m}=Uh^\Dir_{m}$ and $u^\Neu_{m}=Uh^\Neu_{m}$ both satisfy the
  equation
  \begin{equation*}
    -u''=\lambda_m u,
  \end{equation*}
  hence there is a constant (independent of $m$) such that
  \begin{equation*}
     |u^\Dir_{m}(t)| + \left| {u^\Dir_{m}}'(t)\right| 
     \le C \|u^\Dir_{m}\|_{\Lsymb[2]}
       \quad \text{and} \quad
     |u^\Neu_{m}(t)| + \left| {u^\Neu_{m}}'(t)\right|
     \le C \|u^\Neu_{m}\|_{\Lsymb[2]}
  \end{equation*}
  for all $t \in [0,L]$.  We will modify $u^\Dir_{m}$ in order to satisfy
  the Dirichlet boundary condition: for $\eta=\sqrt{\eps|\log\eps|}$
  we define
  \begin{equation*}
    \chi_m(t)=
    \begin{cases}
       t/\eta     & \text{if $t \in [0,\eta]$,}\\
       1          & \text{if $t \in [\eta,L-\eta]$,}\\
       (L-t)/\eta & \text{if $t \in [L-\eta,L]$}\\
    \end{cases}
  \end{equation*}
  and we define $\psi_m$ via
  \begin{equation*}
    U \psi_m = \chi_m u^\Dir_{m}+u^\Neu_{m}.
  \end{equation*}
  We have
  \begin{equation*}
    \|\phi_m -\psi_m\|_{\Lsymb[2]}^2 
    = \|k_m\|^2_{\Lsymb[2]} + \|(1-\chi_m)u^\Dir_{m}\|^2_{\Lsymb[2]}
  \end{equation*}
  hence
  \begin{equation*}
    \lim_{m\to\infty} \|\phi_m-\psi_m\|_{\Lsymb[2]}=0.
  \end{equation*}
  Moreover
  \begin{equation*}
    \int_0^L |(U\psi_m)'(t)|^2dt 
    =   \int_0^L |(\chi_m u^\Dir_{m})'(t)|^2 dt
      + \int_0^L |{u^\Neu_{m}}'(t)|^2 dt.
  \end{equation*}
  But
  \begin{align*}
    \int_0^L |(\chi_m u^\Dir_{m})'(t)|^2 dt
    & = \int_0^L |\chi_m'(t)|^2|u^\Dir_{m}(t)|^2dt 
      + \int_0^L \Bigiprod {\frac{d}{dt} \bigl(\chi_m^2u^\Dir_{m}\bigr)} 
                           {\frac{d}{dt}u^\Dir_{m}} dt\\
    &= \int_0^L |\chi_m'(t)|^2|u^\Dir_{m}(t)|^2dt
     + \lambda_m \int_0^L |\chi_m u^\Dir_{m}(t)|^2 dt
  \end{align*}
  and
  \begin{equation*}
    \int_0^L |\chi_m'(t)|^2|u^\Dir_{m}(t)|^2dt
    = O(\sqrt{\eps|\log\eps|}).
  \end{equation*}
  Similarly, we have
  \begin{align*}
    \int_0^L |{u^\Neu_{m}}'(t)|^2dt
    &= \lambda_m \int_0^L |u^\Neu_{m}(t)|^2dt 
      +\bigl[u^\Neu_{m}{u^\Neu_{m}}'\bigr]_0^L\\
    &= \lambda_m \int_0^L |u^\Neu_{m}(t)|^2dt+O(\eps|\log\eps|).
  \end{align*}

 Now on $\overline M\setminus (\mcC^+_0\cup \mcC^-_0)$ we set
  \begin{equation*}
    \psi_m=\phi_m
  \end{equation*}
 and on $\mcC^\pm_\eps$, we know that
  \begin{equation*}
    \phi_m=U^\ast(u_{\pm,m}+v_{\pm,m})+\phi_{\pm,m}^\Lambda
  \end{equation*}
  where $u_{\pm,m}$ is described in \lref{lem5} and the corresponding
  assertion on the harmonics parts in \cref{cor:hp}.  In particular,
  $u_{\pm,m}$ has a well defined extension $\overline u_{\pm,m}$ which
  is in the domain of the limit operator. Moreover, we know that
  \begin{equation*}
    \lim_{m\to\infty} \|v_{\pm,m}\|_{\Lsqr{\mcC^\pm_\eps}}=0
  \end{equation*}
  and for a certain constant $C$ we have
  \begin{equation*}
    \|\phi^\Lambda_{\pm,m}\|^2_{\Lsqr{\mcC^\pm_\eps\setminus \mcC^\pm_\eta}}
    \le C \eta^2
  \end{equation*}
  for each $\eta>\eps$.  Moreover
  \begin{equation*}
    UD^2U^\ast u_{\pm,m}=\lambda_m u_{\pm,m}, \quad
    UD^2U^\ast v_{\pm,m}=\lambda_m v_{\pm,m} \quad \text{and} \quad
    D^2 \phi_{\pm,m}^\Lambda=\lambda_m \phi_{\pm,m}^\Lambda.
  \end{equation*}
  We consider two cut off functions
  \begin{equation*}
    \xi_0(t)=
    \begin{cases}
          1    & \text{if $t\ge 1/2$,}\\
          4t-1 & \text{if $t\in [1/4,1/2]$,}\\
          0    & \text{if $t\le 1/4$}
    \end{cases}
  \end{equation*}
  and, with $\eps=\eps_m,$
  \begin{equation*}
    \xi_m(t)=
    \begin{cases}
          1   & \text{if $t \ge 2\sqrt{\eps}$,}\\
          \dfrac{\log(2\eps)-\log(t)}{\log(\sqrt{\eps})} 
              & \text{if $t \in [2\eps,2\sqrt{\eps}]$,}\\
          0   & \text{if $t \le 2\eps$.}
    \end{cases}
  \end{equation*}
  On $\mcC^\pm_0$, we define
  \begin{equation*}
    \psi_m = U^\ast(\overline u_{\pm,m} 
      + \xi_0 v_{\pm,m}) + \xi_m\phi_{\pm,m}^\Lambda.
  \end{equation*}
  There exists a $\delta>0$ such that
  \begin{equation*}
    \|\psi_m\|_{\Lsqr{\mcC^\pm_0\setminus\mcC^\pm_\eps }}
    = O(\eps^\delta).
  \end{equation*}
  Moreover
  \begin{equation*}
    \| \psi_m-\phi_m\|_{\Lsqr{\overline M}}
    \le O(\eps^\delta) 
      + \sum_{s=\pm} 
         \Bigl[ \|v_{s,m}\|_{\Lsqr{\mcC^s_\eps}}
              + \|\phi^\Lambda_{s,m}
                   \|_{\Lsqr{\mcC^s_\eps\setminus\mcC^s_{2\sqrt{\eps}}}}
         \Bigr].
  \end{equation*}
  Hence
  \begin{equation*}
    \lim_{m\to\infty}\| \psi_m-\phi_m\|_{\Lsqr{\overline M}}=0
  \end{equation*}
  and the correspondence $\phi_m \mapsto \psi_m$ is almost isometric.

  We now deal with the quadratic form expression. Namely, we want to show that
  \begin{equation}
    \label{eq:est.fq}
    \|(d+d^*)\psi_m\|^2_{\Lsqr{\overline M}}
    \le \lambda_m \|\psi_m\|^2_{\Lsqr{\overline M}}+o(1).
  \end{equation}
  After an integration by part,we get
  \begin{equation}
    \label{barM}
    \|(d+d^*)\psi_m\|^2_{\Lsqr{\overline M\setminus(\mcC^+_0\cup \mcC^-_0) }}
     = \lambda_m \| \psi_m 
         \|^2_{\Lsqr{\overline M\setminus(\mcC^0_0\cup \mcC^-_0)}}
       +BT
   \end{equation}
   where $BT$ is a certain boundary integral over the regular part of
   $\partial\mcC^+_0\cup\partial \mcC^-_0$.  Indeed the behaviour of
   $\overline u_{s,m}$ implies that
   \begin{equation}
     \label{baru}
     \|(d+d^*) \overline u_{\pm,m} \|^2_{\Lsqr{\mcC^\pm_0}}
     = \lambda_m \|\overline  u_{\pm,m} \|^2_{\Lsqr{\mcC^\pm_0}}+BT_{\pm,u}
   \end{equation}
   where $BT_{\pm,u}$ is a certain boundary integral over
   $\partial\mcC^\pm_0$.  Similarly, we have
  \begin{align}
    \nonumber
    \|(d+d^*) (\xi_0v_{\pm,m})\|^2_{\Lsqr{\mcC^\pm_0}}&
    = \int_{\mcC^\pm_0} |d\xi_0|^2|v_{\pm,m}|^2 d\vol
    + \bigiprod{(d+d^*)(\xi_0^2v_{\pm,m})} {(d+d^*)v_{\pm,m}}\\ \nonumber
    &= \int_{\mcC^\pm_0} |d\xi_0|^2|v_{\pm,m}|^2 d\vol
    + \lambda_m\int_{\mcC^\pm_0}|\xi_0v_{\pm,m}|^2d\vol+BT_{\pm,v}\\
     \label{v}
    &= \lambda_m\int_{\mcC^\pm_0}|\xi_0v_{\pm,m}|^2d\vol
            +o(1)+BT_{\pm,v}
  \end{align}
   where again $BT_{\pm,v}$ is a certain boundary integral over
  $\partial\mcC^\pm_0$.  Similarly, we get
  \begin{multline}
    \label{Lambda}
        \|(d+d^*) (\xi_m\phi^\Lambda_{\pm,m})\|^2_{\Lsqr{\mcC^\pm_0}}\\
    = \int_{\mcC^\pm_0} |d\xi_m|^2|\phi^\Lambda_{\pm,m}|^2 d\vol
    + \bigiprod{(d+d^*)(\xi_m^2\phi^\Lambda_{\pm,m})} 
                      {(d+d^*)\phi^\Lambda_{\pm,m}}\\
    = \int_{\mcC^\pm_0} |d\xi_m|^2|\phi^\Lambda_{\pm,m}|^2 d\vol
    + \lambda_m\int_{\mcC^\pm_0}|\xi_m \phi^\Lambda_{\pm,m}|^2 d\vol
    + BT_{\pm,\Lambda}
  \end{multline}
  where $BT_{\pm,\Lambda}$ is a certain boundary integral over
  $\partial\mcC^\pm_0$.  Furthermore, we set
  $M(r)=\|\phi^\Lambda_{\pm,m}\|^2_{\Lsqr{\mcC^\pm_\eps\setminus\mcC^\pm_r}}$.
  By \lref{lem:high.ev}, $M(r)$ is of order
  $O(\frac{r^2}{\Lambda^2})$, and we have
  \begin{align*}
    \int_{\mcC^\pm_0} |d\xi_m|^2|\phi^\Lambda_{\pm,m}|^2 d\vol
    &= \frac{4}{|\log\eps|^2} 
            \int_{2\eps}^{2\sqrt{\eps}} \frac{1}{r^2} \, dM(r)\\
    &= \frac{4}{|\log\eps|^2}
       \bigg[ 
                \frac{M(2\sqrt{\eps})}{4\eps}
              - \frac{M(2\eps)}{4 \eps^2}
              + 2 \int_{2\eps}^{2\sqrt{\eps}} \frac{M(r)}{r^3} dr
       \bigg]\\
    &=O\bigg(\frac{1}{|\log\eps|}\bigg).
  \end{align*}
  We also have
  \begin{equation*}
    BT+BT_{+,u}+BT_{-,u}+BT_{+,v}+BT_{-,v}+BT_{+,\Lambda}+BT_{-,\Lambda}=0
  \end{equation*}
  and the square of the $\Lsymb[2]$-norm of $(d+d^*)\psi_m$ on
  $\overline M$ is the sum of~\eqref{barM}--\eqref{Lambda}. Hence we
  obtain~\eqref{eq:est.fq}.

  The argument for the first $N$ eigenvalues is as follows: Let
  $\phi^k_m:=\phi^k_m(\eps_m)$, $k=1,\dots, N$, be an orthonormal family of
  eigenforms for the eigenvalues $\lambda_k(\eps_m)$ (we drop here the
  index $p$) such that
  $\lambda_1(\eps_m)\leq\dots\leq\lambda_N(\eps_m)$ and
  $\lim_{m\to\infty}\lambda_N(\eps_m)=\liminf_{\eps\to
    0}\lambda_N(\eps)$ for $\lim_{m\to\infty}\eps_m=0$. We have just
  seen that to each $\phi^k_m$ we have associated a $\psi_m^k$  in the domain of the 
  limit quadratic form. 
    Then the fact that the map $\phi_m\to\psi_m$ is almost an isometry,
     shows that
  \begin{equation*}
    | \iprod {\psi^k_m(\eps_m)} {\psi^l_m(\eps_m)} 
           -\delta(k,l)| =o(1)
  \end{equation*}
  as $m$ tends to infinity  for all $k, l$, where $\delta(k,l)$
  denotes the Kronecker symbol.
  
  Now if we calculate the Rayleigh-Ritz quotient for an element $\psi$
  of the vector space with base $\{\psi^k_m,\, k=1,\dots, n\}$, it
  follows from the two preceding estimates and~\eqref{eq:est.fq}
  applied for each $\psi^k_m$ that
  \begin{equation*}
    \|(d+d^\ast)\psi\|^2\leq (\lambda_N(\eps_m)+o(1))\|\psi\|^2.
  \end{equation*}
  The conclusion follows then from the minimax formula, namely
  $\mu_N\leq\lambda_N(\eps_m)+o(1)$ for all $m\in\N$ and at the limit:
  $\mu_N\leq\liminf_{\eps\to 0}\lambda_N(\eps)$.
\end{proof}

\section{Covering manifolds}
\label{sec:covering}
In this section we explain how the convergence argument of the
preceding section can be used also for a covering manifold in order to
show the existence of spectral gaps. Let us first describe the
covering manifold and the \emph{Floquet decomposition} of a periodic
operator on the covering.

Let $X$ be an $(n+1)$-dimensional non-trivial covering manifold, with
quotient $M$ and covering group $\Z$.  This covering defines a
non-trivial element $c\in H^1(M,\Z).$ To each element of $H^1(M,\Z)$
corresponds a homotopy class of functions $f_c:M\to \Sphere^1$ and if
$c\neq 0$ then $f_c$ is surjective. It can be chosen smooth, so we
know, by the Sard's theorem, that $f_c$ has a regular value $y$.
Therefore, $\Sigma=f_c^{-1}(y)$ is a hypersurface of $M$ such that
$F:=M \setminus \Sigma$ is a fundamental domain for $X$.  Let
$\{g_\eps\}_\eps$ be the family of metrics on $M$ constructed in
\sref{sec:geo}. We denote the lift of $g_\eps$ onto $X$ also by
$g_\eps$.

Let $\chi \in \hat \Z$ be a character of the group $\Z$, i.e., a group
homomorphism $\map \chi \Z \Sphere^1$. Clearly, such a homomorphism is
given by $\chi(\gamma)=\e^{\im \gamma \theta}$ for some $\theta \in
[0,2\pi]$. We will identify $\chi$ and $\theta$ in the sequel.

We can associate a complex line bundle $E_\theta^0 \to M$ to the
$\Z$-covering $X \to M$ since the covering $X \to M$ is a principal
bundle with discrete fibre $\Z$. Similarly, we denote by $E_\theta^p
\to M$ the bundle associated to the $\Z$-covering $\Lambda^p T^*X \to
\Lambda^p T^*M$. A smooth section $\omega$ in $E^p_\theta$ can be
considered as a smooth section in $\Lambda^p T^*X$ satisfying the
so-called \emph{equivariance condition}
\begin{equation}
  \label{eq:eq.var}
  \omega(\gamma + x) = \e^{\im \gamma \theta} \omega(x)
\end{equation}
for $x \in X$ and $\gamma \in \Z$ where we write the action of $\Z$ on
$X$ additively. Clearly, such a section is determined on a fundamental
domain $F \subset X$. The $\Lsymb[2]$-space of $\theta$-equivariant
sections with respect to the metric $g$ will be denoted by
$\Lsqr{E^p_\theta,g_\eps}$. Since $\Lsymb[2]$-functions do not
``feel'' the condition~\eqref{eq:eq.var} on a fundamental domain,
$\Lsqr{E^p_\theta,g_\eps}$ is unitarily equivalent to $\Lsqr{\Lambda^p
  T^*F,g_\eps}$, independently of $\theta$.

Using Floquet theory (see e.g.~\cite[XIII.16]{reed-simon-4}), the
$\Lsymb[2]$-space of forms on $(X,g_\eps)$ can be transformed into
\begin{equation}
  \label{eq:dir.int.hs}
  \Lsqr{\Lambda^p T^*X, g_\eps} 
  \cong \int_{\hat \Z} \Lsqr{E^p_\theta, g_\eps}\, d\theta.
\end{equation}
The Gau\ss -Bonnet operator $D$ acting on $(X,g_\eps)$ can be decomposed
under this direct integral representation as
\begin{equation}
  \label{eq:dir.int.op}
  D \cong \int_{\hat \Z} D_\theta \, d\theta
\end{equation}
where the domain of $D_\theta$ consists of those forms $\omega$ having
a $\theta$-equivariant continuation in $\Sobloc X$.  For our purposes,
it will be convenient to use the fundamental domain corresponding to
$F= M_\eps \setminus \{2\} \times \Sigma$, i.e., we cut along the
right end of the collar neighbourhood $\mcU={]}{-}2,2{[} \times
\Sigma$.  The domain of $D$ is then given by forms $\omega$, such that
their components are piecewise in $\Sobsymb[1]$ and satisfy the
boundary conditions
\begin{equation}
  \label{eq:recolfq.th}
  \omega_-= \e^{\im \theta} \omega_+
\end{equation}
where $\omega_-$ denotes the limit of $\omega$ on $\{2\}
\times \Sigma \subset \clo \mcU$ and $\omega_+$ the limit from the
opposite side $M \setminus \mcU$.

The Hodge-de Rham operator $\Delta^p_\eps=D^2$ acting on $p$-forms on
$(X,g_\eps)$ decomposes similarly, where the domain of
$\Delta^p_{\eps,\theta}=D_\theta^2$ consists of those forms $\omega$
such that their components are piecewise in $\Sobsymb[2]$ and satisfy
additionally to~\eqref{eq:recolfq.th} the first order boundary
conditions
\begin{equation}
  \label{eq:recolop.th}
  \omega_-' = - \e^{\im \theta} \omega_+',
\end{equation}
where $\omega_-'$ denotes the outward normal derivative of $\omega$
on $\{2\} \times \Sigma \subset \clo \mcU$ and similarly, $\omega_+'$
the outward normal derivative from the opposite side.

The spectrum of the Hodge-de Rham operator $\Delta^p_{\eps,\theta}$ is
purely discrete and will be denoted by $\lambda_{k,\theta}^p(\eps)$,
ordered in increasing order and repeated according to the
multiplicity.  From the direct integral representation (and the
continuous dependence on $\theta$) it follows that the spectrum of the
Hodge-de Rham operator $\Delta_\eps^p$ on $X$ is given as
\begin{equation}
  \label{eq:spec}
  \spec {\Delta_\eps^p} = \bigcup_{k \in \N} B_k^p(\eps)
  \qquad \text{where} \qquad
  B_k^p(\eps) 
  = \set {\lambda_{k,\theta}^p(\eps)}{\theta \in [0, 2\pi]}
\end{equation}
are compact intervals, called \emph{bands}.

Our convergence result \tref{thmC} holds also for the
$\theta$-equivariant eigenvalues $\lambda_{k,\theta}^p(\eps)$.
Although we have shown this convergence only for $\theta=0$, all
arguments remain the same noting that the arguments are local in
$\mcV$ or rely on elliptic regularity elsewhere. Let $\Lambda>0$, then
by continuity of the map $\theta\mapsto \lambda_{k,\theta}^p(\eps)$,
we know that there is some $\theta^\pm_\eps$ such that
\begin{equation*}
   B_k^p(\eps)\cap[0,\Lambda] = 
      [\lambda_{k,\theta^-_\eps}^p(\eps),
                    \lambda_{k,\theta^+_\eps}^p(\eps)]
\end{equation*}
provided $\lambda^p_k(0)< \Lambda$ and $\eps>0$ small enough.
Applying the preceding convergence result to
$\lambda_{k,\theta^-_\eps}^p(\eps)$ and
$\lambda_{k,\theta^-_\eps}^p(\eps)$, we obtain that
\begin{equation*}
  \lim_{\eps\to 0} \lambda_{k,\theta^\pm_\eps}^p(\eps)=\lambda_k^p(0)
\end{equation*}
where $\lambda_k^p(0)$ denotes the spectrum of the limit operator on
$p$-forms.
 
Hence the limit does no longer depend on the Floquet parameter $\theta$. This
means, that the bands $B_k^p(\eps)$ shrink to a point
$\{\lambda_k^p(0)\}$, where $\lambda_k^p(0)$ denotes the spectrum of
the limit operator.

We therefore have shown our main result (remind that $n+1$ is the
dimension of $X$):
\begin{thm}
  \label{thm:gaps}
  Assume that $n$ is odd or $H^{n/2}(\Sigma)$ is trivial.  Given $N
  \in \N$ there is a metric $g=g_N$ such that the Hodge-de Rham
  operator on the $\Z$-covering $(X,g_N)$ has at least $N$ gaps in its
  (essential) spectrum.

  If $n$ is even and $H^{n/2}(\Sigma)\ne 0$ then the result remains
  true for the Hodge-de Rham operator acting on $p$-forms providing that 
  $p\ne n/2$ and $p \ne n/2+1$.
\end{thm}

  \subsection*{Proof of \tref{thmD}.}
  Let us now have a look at the Dirac operator on a spin manifold $M$.
  It is a consequence of~\cite{Roe} that the spectrum of the Dirac
  operator on the periodic manifold is the whole real line if
  $\alpha_n(\Sigma) \neq 0$. For the other implication, the same
  calculations as before, but with simpler expressions. Let us sketch
  the ideas here. If $\alpha_n(\Sigma)=0$ then, by the result
  of~\cite{ADH}, there exists a metric $h$ on $\Sigma$ such that the
  corresponding Dirac operator has no harmonic spinor.  We endow $M$
  with a metric such that its restriction to $\Sigma$ coincides with
  $h$. Let $\Lambda>0$ be such that the spectrum of the Dirac operator
  $D_0$ on $\Sigma$ does not intersect the interval
  $[-\Lambda,\Lambda]$.  By a scale of the metric $h$ we can always
  suppose that $\Lambda$ is large enough such that the Dirac operator
  $D$ is essentially self adjoint on the limit manifold $\overline M$
  (see \sref{sec:limit}).

  The precise behaviour of the Dirac operator on cones can be found
  in~\cite{Ch}. If $(\mcM_\eps,g_\eps)$ is isometric to $I_\eps \times
  \Sigma$ endowed with the warped product metric
  $d\tau^2+f_\eps(\tau)^2 h$ where $I_\eps = {]}{-}(L/2+1-\eps),
  L/2+1-\eps{[}$, then the Dirac operator on $\mcM_\eps$ is unitarily
  equivalent to
  \begin{equation*}
    \begin{pmatrix}
      0&1\\
     -1&0
    \end{pmatrix} \cdot
    \Big(\partial_\tau+\frac{1}{f_\eps(\tau)}
    \begin{pmatrix}
      0 &-D_0\\
      -D_0 & 0 
    \end{pmatrix}\Big)
  \end{equation*}
  on $\mcM_\eps$ using the isometry $\map U {\Lsqr
    {\mcM_\eps,g_\eps}}{\Lsqr{I_\eps,\Lsqr{\Sigma,h}}}$ as in
  \sref{sec:lap}.  Here, $f_\eps$ can be chosen either continuous and
  piecewise smooth as before, or smooth on the whole interval by the
  argument described in \sref{sec:geo}.  Anyway, we can redo the
  previous calculus with $A=A_0$, and there is no more boundary term
  in the expression of the quadratic form~\eqref{locform}
  or~\eqref{eq:est1}.

  For $\eps_m \to 0$, let $\phi_m$ be a family of eigenspinors on
  $M_{\eps_m}$ corresponding to the eigenvalues $\lambda_{\eps_m} \to
  \lambda$. Due to our choice of $h$ and $\Lambda$, the
  decomposition~\eqref{decomp} of the eigenspinor $\phi_m$ on
  $\mcM_{\eps_m}$ is reduced to the last term, and \lref{lem:high.ev}
  applies directly to $\phi_m$: There exists a constant $C>0$ such that
  \begin{equation*}
    \|\phi_m\|_{\Lsqr {\mcA_\eps}}^2
    \le C \frac{\eps^2}{\Lambda^2}
    \qquad \text{and} \qquad
    \|\phi_m\|^2_{\Lsqr {\mcC^\pm_\eps\setminus \mcC^\pm_{\eta}}}
    \le C\frac{\eta^2}{\Lambda^2}
  \end{equation*}
  as soon as $\eps_m\leq \eta$.  Thus, the $\Lsymb_2$-norm of the
  eigenspinors on the handle converges to $0$. Moreover, the limit
  spectrum will consist only on the spectrum of the Dirac operator
  with \emph{minimal} domain $D_{\min}$ on $\overline M$. The proof of
  \tref{thmC} can now be followed verbatim: for the `upper bound', the
  proof is reduced to the easiest part, namely eigenspinors in
  $\dom(D_{\min})$, and for the `lower bound' we use the cut-off
  function $\xi_m(t)$ on the cones defined there.

  The limit spectrum is the same for the operator involving the
  Floquet parameter.  Finally, the result of \tref{thmD} follows.

\section{Harmonic forms and small eigenvalues}
\label{sec:harm.forms}
Returning to the situation of \sref{sec:limit}, we can ask for the
multiplicity of the zero eigenvalue, which is given by the cohomology.
The calculation made there shows that ``small eigenvalues'' can occur,
{\it i.e.} $\lambda_\eps\neq 0$ such that $\lim_{\eps\to 0}
\lambda_\eps=0$.

We suppose here that if $n$ is even then $H^{n/2}(\Sigma)=0$ and if
$n$ is odd that $\Delta_\Sigma$ has no eigenvalue in $]0,1[$ then the
only limit operator involved is $D_{\max}\circ D_{\min}$, and we know
by the works of Cheeger that the kernel of $D_{\max}\circ D_{\min}$
coincides with the intersection cohomology of the manifold with
conical singularities.  Let $N$ be the number of small, or null
eigenvalues.  By the precedent result we know that
\begin{align*}
   N &= \dim I\!H^p(\overline M)+\dim H^{p-1}(\Sigma) &&
       \text{for $p<(n+1)/2$,}\\
   N &= \dim I\!H^p(\overline M)+\dim H^{p}(\Sigma)   &&
       \text{for $p>(n+1)/2$,}\\
   N &= \dim I\!H^p(\overline M)+\dim H^{p-1}(\Sigma)+\dim H^{p}(\Sigma)&&
       \text{for $p=(n+1)/2$.}
\end{align*}
 
The manifold $M_\eps$ is covered by the two open sets $U_0=M \setminus
({]}{-}1,1{[}\times \Sigma)$ and the collar $U={]}{-}2,2{[}\times
\Sigma$. The Mayer-Vietoris argument gives then a long exact sequence
\begin{equation*}
  \dots \to H_{\mathrm c}^q(U)\stackrel{j}{\to} 
     H^q(M_\eps)\stackrel{r}{\to}H^q(U_0)\to H_{\mathrm c}^{q+1}(U)
  \to \dots
\end{equation*}
But since $U$ is a cylinder, $H_{\mathrm c}^q(U)=H^{q-1}(\Sigma)$ for
all $q$. On the other hand $H^q(U_0)=I\!H^q(\overline M)$ for $q\leq
n/2$, the long exact sequence gives then that
\begin{equation*}
  \dim H^q(M_\eps)\leq\dim I\!H^q(\overline M)+\dim H^{q-1}(\Sigma)
\end{equation*}
and the equality is obtained if and only if $r$ in surjective and $j$ is injective.

\emph{So, for $p<(n+1)/2$ there are small eigenvalues as
  soon as $j$ is not injective or $r$ is not surjective.  In
  particular for $p=0$ the three spaces $I\!H^0(\overline
  M),\,H^0(U_0)$ and $H^0(M_\eps)$ are isomorphic to $\R$ and there is
  no small eigenvalue.}

For $p>(n+1)/2$ we use that $I\!H^q(\overline M)=H_{\mathrm c}^q(U_0)$
for $q\geq 1+n/2$ so we look at the long exact sequence
\begin{equation*}
  \dots \to H_{\mathrm c}^q(U_0)
  \stackrel{j}{\to} H^q(M_\eps)
  \stackrel{r}{\to}H^q(U)
  \to H_{\mathrm c}^{q+1}(U_0)
  \to \dots
\end{equation*}
and use the identity $H^q(U)=H^q(\Sigma)$.

For $p=(n+1)/2$ we have to look at the more complicate diagram
\begin{alignat*}{2}
  \dots \to H^{\frac{n-1}{2}}(\Sigma)\stackrel{\delta}{\to} & 
   H_{\mathrm c}^{\frac{n+1}{2}}(U_0) \stackrel{j}{\to} 
        H^{\frac{n+1}{2}}(M_\eps) \stackrel{r}{\to} & 
   H^{\frac{n+1}{2}}(\Sigma)&\to H_{\mathrm c}^{\frac{n+3}{2}}(U_0) 
      \to\dots\\ 
   \Big\downarrow\! \iota \hspace{.5cm} 
          \circlearrowleft &\hspace{.5cm}\shortparallel & &\hspace{1cm}
   \shortparallel\\
   \dots \to H^{\frac{n-1}{2}}(\Sigma_-\cup\Sigma_+)
             \stackrel{\overline\delta}{\to} & 
   H_{\mathrm c}^{\frac{n+1}{2}}(U_0) \twoheadrightarrow 
          I\!H^{\frac{n+1}{2}}(\overline M)\to &  
   0\hspace{.5cm}&\to H_{\mathrm c}^{\frac{n+3}{2}}(U_0) \to \dots
\end{alignat*}
Here $\iota(\omega)=(\omega,\omega)\in
H^{\frac{n-1}{2}}(\Sigma_-\cup\Sigma_+)=
\bigl(H^{\frac{n-1}{2}}(\Sigma)\bigr)^2$.  The long exact sequence gives
then
\begin{align*}
  \dim H^{\frac{n+1}{2}}(M_\eps) & 
  \leq \dim H^{\frac{n+1}{2}}(\Sigma) +
       \dim H_{\mathrm c}^{\frac{n+1}{2}}(U_0)-\dim \text{Rg} (\delta)\\&
  \leq \dim H^{\frac{n+1}{2}}(\Sigma) +
       \dim I\!H^{\frac{n+1}{2}}(\overline M) +
     \dim \text{Rg} (\overline\delta)-\dim \text{Rg} (\delta).
\end{align*}
But $\dim \text{Rg} (\overline\delta)-\dim \text{Rg} (\delta)\leq \dim
H^{\frac{n-1}{2}}(\Sigma) $ and the equality
\begin{equation*}
  \dim H^{\frac{n+1}{2}}(M_\eps)
  = \dim H^{\frac{n+1}{2}}(\Sigma)
  + \dim H^{\frac{n-1}{2}}(\Sigma)
  + \dim I\!H^{\frac{n+1}{2}}(\overline M)
\end{equation*}
holds if and only if $r$ is surjective and $\dim \text{Rg}
(\overline\delta)=\dim {Rg} (\delta)+\dim H^{\frac{n-1}{2}}(\Sigma)$,
this last relation means that $\ker
\overline\delta\subset\iota(H^{\frac{n-1}{2}}(\Sigma))$.


\begin{thebibliography}{Am56}
 \bibitem{ADH} B.~Ammann, M.~Dahl, and E.~Humbert, \emph{Surgery and
     harmonic spinors}, Preprint \texttt{arXiv:math.DG/0606224}.
 \bibitem{AC1} C.~Ann\'e and B.~Colbois, \emph{{Op\'erateur de
       Hodge-Laplace sur des vari\'et\'es compactes priv\'ees d'un
       nombre fini de boules}}, J. Funct.  Anal. \textbf{115} (1993),
   190--211.

 \bibitem{AC2} C.~Ann\'e and B.~Colbois, \emph{{Spectre du laplacien
       agissant sur les $p$-formes diff\'e\-ren\-tielles et
       \'ecrasement d'anses}}, Math.  Ann. \textbf{303} (1995), no.~3,
   545--573.

 \bibitem{APS3} M. F.~Atiyah,  V. K.~Patodi, and I. M.Singer,
     \emph{Spectral asymmetry and {R}iemannian geometry. {III}},
   {Math. Proc. Cambridge Philos. Soc.}, \textbf{79}, ({1976}), no.~1, {71--99}.


 \bibitem{B} C.~B{\"a}r, \emph{The {D}irac operator on hyperbolic
     manifolds of finite volume}, J. Differential Geom.  \textbf{54}
   (2000), no.~3, 439--488.

 \bibitem{BMS} N.~V. Borisov, W.~M{\"u}ller, and R.~Schrader,
   \emph{Relative index theorems and supersymmetric scattering theory},
   Comm. Math. Phys. \textbf{114} (1988), no.~3, 475--513.

 \bibitem{BS} J.~Br{\"u}ning and R.~Seeley, \emph{An index theorem for
     first order regular singular operators}, Amer. J. Math.
   \textbf{110} (1988), no.~4, 659--714.

 \bibitem{car} G.~Carron, \emph{A topological criterion for the
     existence of half-bound states}, J. London Math. Soc. (2)
   \textbf{65} (2002), no.~3, 757--768. 


 \bibitem{C} J.  Cheeger, On the Hodge theory of Riemannian
   pseudomanifolds.  \emph{Geometry of the Laplace operator,
     Honolulu/Hawai 1979, Proc. Symp. Pure Math., Vol. {\bf 36}} (1980)
   91-146.

 \bibitem{CGF} J. Cheeger, K. Fukaya, and M. Gromov, \emph{Nilpotent
     structures and invariant metrics on collapsed manifolds}, J. Amer.
   Math. Soc. \textbf{5} (1992), 327--372.
 
 \bibitem{Ch} A.W.  Chou, \emph{The Dirac operator on spaces with
     conical singularities and positive scalar curvature.} Trans. Amer.
   Math. Soc. {\bf 289} (1985) 1--40.
  
 \bibitem{CC} B. Colbois, G.  Courtois, \emph{Convergence de
     vari\'et\'es et convergence du spectre du Laplacien}, Ann. Sci.
   \'Ecole Norm.  Sup.\ (4) {\bf 24} (1991) 507--518.

 \bibitem{D} J. Dodziuk, \emph{Eigenvalues of the Laplacian on forms},
   Proc. Amer. Math. Soc. {\bf 85} (1982), 437--443.

 \bibitem{Di} J. Dieudonn\'e, \emph{Calcul infinit\'esimal.} Hermann,
   Paris (1968).
  
 \bibitem{HM} E.~Hunsicker and R.~Mazzeo, \emph{Harmonic forms on
     manifolds with edges.}  Int. Math. Res. Not.\ {\bf 52} (2005)
   3229--3272.
 
 \bibitem{L} M. Lesch, \emph{Operators of Fuchs type, conical
     singularitites, and asymptotic methods.}  Teubner-Texte zur
   Mathematik 136, Stuttgart (1997).

 \bibitem{Lott} J. Lott, \emph{Collapsing and the differential form
     Laplacian: the case of a smooth limit space.} Duke Math.\ J.\ {\bf
     114} (2002), no. 2, 267--306.

 \bibitem{Mac}P.~McDonald, \emph{The Laplacian for spaces with
     cone-like singularities}, Thesis, MIT (1990).

 \bibitem{Maz} R. Mazzeo, \emph{Resolution blowups, spectral
     convergence and quasi-asymptotically conical spaces}, Actes
   Colloque EDP Evian-les-Bains, (2006).

 \bibitem{P}
 O.~Post, \emph{Periodic manifolds with spectral gaps}, J. Diff. Equations
   \textbf{187} (2003), 23--45.

 \bibitem{reed-simon-4} M.~Reed and B.~Simon, \emph{Methods of modern
     mathematical physics IV: Analysis of operators}, Academic Press,
   New York, 1978.

 \bibitem{Roe} J. Roe, \emph{Partitioning non-compact manifolds and the
     dual Toeplitz problem.}  In Operator Algebras and Applications,
   Cambridge University Press. (1989) pp 187-228.

 \bibitem{Row} J. Rowlett, \emph{Spectral geometry and asymptotically
     conic convergence}, Thesis, Stanford (2006).

 \bibitem{RS} D. Ruberman and N. Saveliev, \emph{Dirac operators on
     manifolds with periodic ends}, Preprint
   \texttt{arXiv:math.GT/0702271} (2007).

\bibitem{seeley:92}
R.~Seeley, \emph{Conic degeneration of the {G}auss-{B}onnet operator}, J. Anal.
  Math. \textbf{59} (1992), 205--215.
 \end{thebibliography}

\end{document}